\documentclass[12pt]{article}
\usepackage{geometry}                
\usepackage{mathtools}

\usepackage{enumerate}
\usepackage{graphicx}
\usepackage{amsmath}
\usepackage{amsthm}
\usepackage{amssymb}
\usepackage{mathrsfs}
\usepackage{epstopdf}
\usepackage{hyperref}
\usepackage{bbm} 
\newcommand{\mf}{\mathfrak}

\newcommand{\Hom}{\operatorname{Hom}}

\newcommand{\Z}{\mathbb{Z}}
\newcommand{\Q}{\mathbb{Q}}
\newcommand{\R}{\mathbb{R}}
\newcommand{\C}{\mathbb{C}}

\newcommand{\pC}{\mathcal{C}}
\renewcommand{\mod}{\operatorname{mod}}
\newcommand{\End}{\operatorname{End}}

\newcommand{\bgg}{\Delta}
\newcommand{\demazure}{D}

\newcommand{\grhecke}{\mathcal{H}}
\newcommand{\rhecke}{H_n(\Gamma)}
\newcommand{\ahecke}{\mathscr{H}}     
\newcommand{\idhecke}{\ahecke^h}             
\newcommand{\ihecke}{\ahecke^{h_0}}
\newcommand{\dhecke}{\mathbb{H}}      
\newcommand{\whecke}{{}_\lambda \ihecke_\lambda}
\newcommand{\nhecke}{{}^{h_0}\ahecke}

\newcommand{\aA}{\mathscr{A}}         
\newcommand{\dA}{\mathbb{A}}          
\newcommand{\iA}{\aA^{h_0}}                   
\newcommand{\idA}{\aA^h}

\newcommand{\aT}{\mathscr{T}}         
\newcommand{\dT}{\mathfrak{h}}          
\newcommand{\iT}{\aT^{h_0}}                   
\newcommand{\aS}{\mathscr{S}}

\newcommand{\fhecke}{\mathscr{H}^f}   
\newcommand{\lhecke}{\dot{\ihecke}}
\newcommand{\lA}{\dot{\iA}}   
\newcommand{\qhecke}{\ihecke(G)}  
\newcommand{\qheckeh}{\ihecke(H)}
\newcommand{\qA}{\iA}
\newcommand{\lqA}{\lA}
\newcommand{\qT}{\iT}

\newtheorem{theorem}{Theorem}[subsection]
\newtheorem{lemma}[theorem]{Lemma}
\newtheorem{proposition}[theorem]{Proposition}
\newtheorem{corollary}[theorem]{Corollary}
\newtheorem{claim}[theorem]{Claim}

\newtheorem*{definition}{Definition}
\theoremstyle{definition}

\theoremstyle{remark}
\newtheorem{remark}[theorem]{Remark}
\title{Affine Hecke algebras and quiver Hecke algebras}
\author{Rob Denomme}
\date{}                                           
\begin{document}
\maketitle

\section*{Abstract}
We give a presentation of localized affine and degenerate affine Hecke algebras of arbitrary type in terms of weights of the polynomial subalgebra and varied Demazure-BGG type operators. 
We offer a definition of a graded algebra $\grhecke$ whose category of finite-dimensional ungraded nilpotent modules is equivalent to the category of finite-dimensional modules over an associated degenerate affine Hecke algebra.
Moreover, unlike the traditional grading on degenerate affine Hecke algebras, this grading factors through central characters, and thus gives a grading to the irreducible representations of the associated degenerate affine Hecke algebra.
This paper extends the results
\cite[Theorem 3.11]{rouquier-qha}, and \cite[Main Theorem]{brundan-kleshchev} where the affine and degenerate affine Hecke algebras for $GL_n$
are shown to be related to quiver Hecke algebras in type $A$, and also secretly carry a grading.

\section*{Introduction}
The representation theory of affine and degenerate affine Hecke algebras has a rich and continuing history. 
The work of Kazhdan and Lusztig in \cite{kazhdan-lusztig} (c.f. also  \cite{ginzburg}) gives a parametrization and construction of irreducible modules over an affine Hecke algebra with equal parameters which aren't a root of unity. 
This parametrization is in the spirit of the Langlands program, and is carried out by constructing a geometric action of the Hecke algebra on equivariant (co)homology and $K$-theory of various manifolds related to the flag variety.
Moreover, character formulas for irreducible representations are deduced from this theory, giving a satisfactory geometric understanding of the representation theory of such algebras.
Unfortunately, for unequal parameters the geometric approach has not yielded as much progress.

More recently, the categorification of quantum groups has given a renewed interest to the theory of affine Hecke algebras of type $A$.
It is shown in \cite[Main Theorem]{brundan-kleshchev}, \cite[Theorem 3.16, 3.19]{rouquier-2km} that a localization of the affine Hecke algebra $\mathscr{H}_n$ of $GL_n$ at a maximal central ideal is isomorphic to a localization of a quiver Hecke algebra $H_\lambda$ associated to a finite or affine type $A$ Cartan matrix at a corresponding central ideal.
This coincidence has been used in \cite{mcnamara}, \cite{kleschev-ram} to give a new approach to classifying the irreducible representations of these algebras for parameters $q$ which aren't a root of unity, as well as understanding the homological algebra of their representation categories.
This algebraic approach replaces the set-up of the Langlands program with the main categorification result, that quiver Hecke algebras categorify quantum groups \cite{khovanov-lauda, rouquier-2km}.
Even further, this work shows that the affine Hecke algebras of type $A$ carry a secret grading which recovers the quantum variable in the decategorification.
It is a natural, and important question to ask if these techniques may be used for affine Hecke algebras in other types.

This paper defines a graded algebra $\grhecke$ associated to any simply connected semisimple root datum and arbitrary parameters whose category of (ungraded) finite modules with a nilpotence condition is equivalent to the category of finite modules over the associated degenerate affine Hecke algebra (see Section \ref{graded-hecke-algebra}). 
Thus degenerate affine Hecke algebras in other types and with unequal parameters are  secretly graded as are those of type $A$.
The presentation of $\grhecke$ is a natural analogue of a quiver Hecke algebra, but with the symmetric group $\mathfrak{S}_n$ replaced with the Weyl group of the root datum, thus we refer to $\grhecke$ as a quiver Hecke algebra as well.
In this way we generalize \cite[Theorem 3.16]{rouquier-2km}, and give a grading on finite-dimensional irreducible representations of degenerate affine Hecke algebras.
It is unclear if the algebras $\grhecke$ are in fact related to the geometry of quiver-type varieties, but the considerable applications of the the theory of quiver Hecke algebras gives cause for their study.
Moreover, there are many natural questions to be asked about the algebra $\grhecke$. 
Could the graded characters of an irreducible module could be computed from the geometric standpoint mentioned above?
Are there natural graded cyclotomic quotients of $\grhecke$?
Could one give an algebraic parametrization of the irreducible $\grhecke$-modules and offer an algebraic construction of them following the work of \cite{mcnamara}?

It should be noted that this paper uses localizations where other authors, \cite{lusztig}, \cite{savage} use completions. 
Our approach is rooted in finding a graded version of the degenerate affine Hecke algebras, which does not need the machinery of completions, and is perhaps a simpler approach in the first place.

We now briefly summarize the results of the paper.
Fix a simply connected semisimple algebraic group over a field $k$, and let $h_0\in k$ be a parameter.
In this paper we define a locally unital localized quiver Hecke algebra $\qhecke$ associated to a data $G$ with generators and relations. 
The algebra is given as a direct sum over Weyl group orbits in the dual to a maximal torus when $h_0\not=0$ and over Weyl group orbits in the dual space to a Cartan subalgebra when $h_0=0$:
\begin{align*}
\qhecke = \bigoplus_{\Lambda\in \iT / W} \qhecke_\Lambda.
\end{align*}
We associate a data $G$ and parameter $h_0\in k$ to the affine and degenerate affine Hecke algebras $\ahecke, \dhecke$ associated with this group where $h_0=0$ in the degenerate case.
We then define a non-unital localization, $\lhecke$ of $\ahecke, \dhecke$ and produce an isomorphism,
\begin{align*}
\qhecke\xrightarrow{\sim} \lhecke,
\end{align*}
which generalizes \cite[Theorem 3.11, 3.12]{rouquier-qha}.
The graded version of a degenerate affine Hecke algebra, $\grhecke$, is defined as a subalgebra of $\qhecke$ with $h_0=0$, and we give a separate presentation of this algebra with generators and relations.

In the last section \ref{applications}, we define the \emph{quiver Hecke algebra} $\grhecke$ associated to a degenerate affine Hecke algebra $\dhecke$.
This is a graded algebra whose category of finite-dimensional ungraded nilpotent representations is equivalent to that of $\dhecke$.
In fact, we show in this section that every irreducible ungraded representation of $\grhecke$ has a grading whose graded character is invariant under inverting the grading.
In section \ref{weight-induction} we study the representation theory of $\qhecke$ one weight space at a time.
A crucial tool is the PBW-basis given in Theorem \ref{PBW}.
Using this basis along with a few simple results on the action of the Weyl group on the torus we provide two algebraic constructions of all irreducible representations which have a non-zero eigenspace $V_\lambda$ with $\lambda$ a standard parabolic weight, both in the equal and unequal parameters case.
A highlight of this study is the structure of the so-called \emph{weight Hecke algebra}, $\whecke$, which turns out to be a matrix ring for $\lambda$ a standard parabolic.
This recovers and extends a well known result of Rodier in the case that $\lambda$ is $W$-invariant, as well as a result of Bernstein-Zelevinsky in the case that $\lambda$ is regular.
This chapter includes an example algebraic computation of the graded characters of each irreducible representation of a degenerate affine Hecke algebra of type $SL_3$ with a specific central character.

We hope our construction can be used to obtain new and algebraic insights on representations of (degenerate) affine Hecke algebras at unequal parameters, where geometric methods are missing.

\tableofcontents

\section{Affine Hecke algebras}
\subsection{Bernstein's presentation}
We recall the Bernstein presentation of affine Hecke algebras, following \cite{lusztig}. Let 
$(X,Y,R,\check{R},\Pi)$ be a (reduced) simply connected semisimple root datum.
Thus $X,Y$ are finitely generated free abelian
groups in perfect pairing we denote by $\langle , \rangle$. Further, the finite subsets
$R\subset X$, $\check{R}\subset Y$  of roots and coroots are in a given bijection $\alpha\mapsto \check{\alpha}$. 
The set $R$ is invariant under the simple reflections, $s_\alpha\in \operatorname{GL}(X)$, which are  
given by,
\[
s_\alpha(x) = x-\langle x,\check{\alpha}\rangle \alpha.
\]
Similarly, it is required that $\check{R}$ be invariant under $s_{\check{\alpha}}$, defined by,
\[
s_{\check{\alpha}}(y) = y - \langle y,\alpha\rangle \check{\alpha}.
\]
Denote by $W\subset GL(X)$ the finite Weyl group of the system, and $\Pi\subset R$ a root basis.
As the root system is reduced, the only multiples of a root $\alpha$ which are also roots are $\pm \alpha$.
Lastly, the root datum being simply connected means $X$ contains the fundamental weights $\{\omega_\alpha\}_{\alpha\in \Pi}$ defined as follows.
Given $\alpha\in \Pi$ let $\omega_\alpha\in \Q\cdot R\subset \Q \otimes_\Z X$ be defined by $\langle \omega_\alpha,\check{\beta}\rangle = \delta_{\alpha,\beta}$ for all $\beta\in \Pi$.
The assumption that the root system is simply connected simplifies a number of the formulas in \cite{lusztig}, in particular $\check{\alpha}\not\in 2Y$ for any $\alpha\in \Pi$.
This also simplifies the $W$-module structure of the group ring of $X$, as we shall see.

Let $\aA$ be the group ring of $X$,
\begin{align*}
\aA = \Z[e^x]_{x\in X} / (e^x e^{x'} = e^{x+x'}), 
\end{align*}
which is a domain.
Finally, fix a parameter set given by a collection $q_{s_\alpha}=q_\alpha$ of indeterminates indexed by $\alpha\in \Pi$ such that $q_{\alpha}=q_\beta$ whenever the order $m_{\alpha,\beta}$ of $s_\alpha s_\beta$ in $W$ is odd.
As convention, we put $m_{\alpha,\alpha} = 2$.

With this data we associate the affine Hecke algebra, $\ahecke$, which appears naturally in the 
complex, admissible representation theory of the associated algebraic group
over $p$-adic fields. 

\begin{definition}
Let $\fhecke$ denote 
the finite Hecke algebra associated to Weyl group of the root datum. This is the $\Z[q_\alpha]_{\alpha\in \Pi}$-algebra generated by symbols $T_\alpha=T_{s_\alpha}, \alpha\in \Pi$, with the relations:
\begin{enumerate}[i.]
\item $\cdots T_\beta T_\alpha = \cdots T_\alpha T_\beta,\quad \text{for $\alpha\not=\beta$, with $m_{\alpha,\beta}$ terms on both sides,}$
\item $(T_\alpha+1)(T_\alpha -q_\alpha)=0, \quad \alpha\in \Pi$.
\end{enumerate}
Denote by $\ahecke$ the affine Hecke algebra of the root system.
As an additive group,
\[
\ahecke = \fhecke\otimes_\Z\aA.
\]
Let $\fhecke$ and $\aA$ be subrings,
with the indeterminates $q_\alpha$ central and give $\ahecke$ the following commutativity relation between $T_\alpha\in \fhecke$, $f\in \aA$:
\begin{equation}\label{commute-relation}
T_\alpha f - s_\alpha(f)T_\alpha= (q_\alpha-1)\dfrac{f - s_\alpha(f)}{1-e^{-\alpha}}.
\end{equation}
\end{definition}
We remark that while $(1-e^{-\alpha})^{-1}\not\in\aA$, 
the fraction appearing on right side of the above commutativity formula is in $\aA$.
For example,
\begin{align*}
\dfrac{e^x-e^{s_\alpha(x)}}{1-e^{-\alpha}} =& \begin{cases}
e^x+ e^{x-\alpha}+\cdots +e^{s_\alpha(x)+\alpha} & \langle x, \check{\alpha}\rangle > 0 \\
-(e^{x+\alpha}+ e^{x+2\alpha}+\cdots +e^{s_\alpha(x)}) & \langle x, \check{\alpha}\rangle < 0 \\
0 & \langle x, \check{\alpha}\rangle = 0.
\end{cases}
\end{align*}

\subsection{Degenerate and interpolating affine Hecke algebras}
The degenerate affine Hecke algebra $\dhecke$ is introduced in this section, along with an algebra $\idhecke$ which interpolates the affine and degenerate affine Hecke algebras.
Let $(X,Y,R,\check{R},\Pi)$ be a simply connected semisimple root datum. 
A set of parameters for the degenerate affine Hecke algebra, $\dhecke$, is a collection of indeterminates $c_\alpha$ such that $c_\alpha=c_\beta$ when $m_{\alpha,\beta}$ is odd.
Let $\pC = \Z[c_\alpha]_{\alpha\in \Pi}$ be the parameter ring. 
Let $\dA = S_\Z(X)$ the symmetric algebra of $X$ over $\Z$, a polynomial algebra over $\Z$ with
variables given by a basis of $X$. 

\begin{definition}
As an additive group let
\begin{align*}
\dhecke = \Z[W]\otimes_\Z \pC\otimes_\Z\dA.
\end{align*} 
Here $\Z[W]$ denotes the group ring of the Weyl group,
generated by the simple reflections $s_\alpha\in W, \alpha\in \Pi$. 
Let $\C[W], \pC$ and 
$\dA$ be subrings of $\dhecke$, the parameters $c_\alpha$ be central and give $\dhecke$ the following commutativity relation 
between $s_\alpha\in \Z[W]$ and $x\in X\subset \dA$:
\begin{align*}
s_\alpha \cdot x - s_\alpha(x)\cdot s_\alpha = c_\alpha\dfrac{x-s_\alpha(x)}{\alpha}.
\end{align*}
\end{definition}

We stop here to remark that as before the fraction on the right of the above formula does define an element of $\dA$.
Indeed,
\begin{align*}
\dfrac{x-s_\alpha(x)}{\alpha} = \langle x, \alpha\rangle .
\end{align*}

The interpolating Hecke algebra, $\idhecke$, is an algebra defined using a parameter $h$
such that the specialization at zero gives, $\idhecke\otimes_{\Z [h]}\Z[h]/(h)\cong \dhecke$ whereas the specialization away from zero gives, $\idhecke\otimes_{\Z[h]}\Z[h^{\pm 1}]\cong \ahecke\otimes_{\Z}\Z[h^{\pm 1}]$,
where the parameters $q_s, c_s$ for $\ahecke$ and $\dhecke$ are related by $q_s=1+hc_s$.

Consider the polynomial ring, $\hat{\idA}$, over $\Z[h]$ with generators
$\{P_x\mid x\in X\}$. The symmetric algebra, $\dA = S_{\Z}(X)$, of $X$ over $\Z$ is the quotient of $\hat{\idA}$ by the relations $P_x + P_{y} = P_{x+y}$, and $h=0$.
Let $\idA$ be the quotient of $\hat{\idA}$ by the relations
\begin{align*}
P_x + P_y +h P_x P_y = P_{x+y}, 
\end{align*}
\begin{align*}
P_0 = 0.
\end{align*}
It is noted in \cite{savage} that $\idA$ is a rational form of the formal group ring of the multiplicative formal group law over the abelian group $X$.
Upon specializing at $h=0$,  $\idA\otimes_{\Z[h]}\Z[h]/(h)\cong S_\Z (X)$. 
Let,
\begin{align*}
U_x = 1+h P_x.
\end{align*}
Notice that,
\begin{align*}
U_x U_y =& h^2 P_x P_y + h (P_x+P_y) +1, \\
=& h P_{x+y}+1, \\
=& U_{x+y}.
\end{align*}
It follows that $\idA$ is isomorphic to the $\Z[h]$-subalgebra of $\aA\otimes_{\Z}\Z[h^{\pm 1}]$ generated by $\{P_x = h^{-1}(e^x-1)\}_{x\in X}$.
Hence, $\idA[h^{-1}]$ is isomorphic to the group ring of $X$ over $\Z[h^{\pm 1}]$.
Note that $W$ acts $\Z[h]$-linearly on $\idA$ and this action specializes to the action of $W$ on $\aA$ and $\dA$.

Let $\pC$ be the parameter ring, $\pC = \Z[c_\alpha]_{\alpha\in \Pi}$ and
let $q_\alpha = 1+hc_\alpha\in \pC[h]$ be parameters for the finite Hecke algebra $\fhecke$ over $\pC[h]$.
We now define the \emph{interpolating hecke algebra}, $\idhecke$.

\begin{definition}
As an additive group let $\idhecke$ be the tensor product, 
\begin{align*}
\idhecke = \fhecke \otimes_{\Z[h]} \idA.
\end{align*}
Let $\idA$ and $\fhecke$ be subalgebras, let $h, c_\alpha$ be central and give $\idhecke$ the following commutativity relation:
\begin{align*}
T_\alpha P_x - P_{s_\alpha(x)} T_\alpha = \begin{cases}
0 
& \text{if $\langle x,\check{\alpha}\rangle = 0$}, \\
c_\alpha \left(\langle x,\check{\alpha}\rangle + h
(P_x+ P_{x-\alpha}+\dots P_{s_\alpha(x)+\alpha}) \right) 
& \text{if $\langle x,\check{\alpha}\rangle > 0$}, \\
c_\alpha \left(\langle x,\check{\alpha}\rangle - h
(P_{x+\alpha}+ P_{x+2\alpha}+\dots P_{s_\alpha(x)}) \right) 
& \text{if $\langle x,\check{\alpha}\rangle < 0$}.
\end{cases}
\end{align*}
\end{definition}

\begin{proposition}
We have canonical identifications $\idhecke\otimes_{\Z[h]}\Z[h]/(h)\cong\dhecke$, and $\idhecke\otimes_{\Z[h]} \Z[h^{\pm 1}]\cong \ahecke\otimes_{\Z}\Z[h^{\pm 1}]$.

\end{proposition}
\begin{proof}
From the above relations we find $\idhecke\otimes_{\Z[h]}\Z[h]/(h)\cong \dhecke$ by sending $P_x$ to the associated element $x$ of the symmetric algebra. 
For the specialization with $\pC[h^{\pm 1}]$, we note that for $\langle x,\check{\alpha}\rangle >0$, the number of terms in the sum $P_x+P_{x-\alpha}+\cdots +P_{s_\alpha(x)+\alpha}$ is precisely $\langle x,\check{\alpha}\rangle$. 
Similarly, for $\langle x,\check{\alpha}\rangle <0$ the number of terms in $P_{x+\alpha}+P_{x+2\alpha}+\cdots +P_{s_\alpha(x)}$ is precisely $-\langle x,\check{\alpha}\rangle$. 
Using the fact that $U_y = 1+h P_y$, as well as $q_\alpha-1=hc_\alpha$ we see,
\begin{align*}
T_\alpha U_x - U_{s_\alpha(x)} T_\alpha = \begin{cases}
0 
& \text{if $\langle x,\check{\alpha}\rangle = 0$}, \\
(q_\alpha-1) \left(
U_x+ U_{x-\alpha}+\dots U_{s_\alpha(x)+\alpha} \right) 
& \text{if $\langle x,\check{\alpha}\rangle > 0$}, \\
-(q_\alpha-1) \left(
U_{x+\alpha}+ U_{x+2\alpha}+\dots U_{s_\alpha(x)} \right) 
& \text{if $\langle x,\check{\alpha}\rangle < 0$}.
\end{cases}
\end{align*}

These are nothing more than the commutativity relations for the affine Hecke
algebra $\ahecke$. 
\end{proof}

\subsection{Demazure operators and polynomial representations}
\label{lusztig-rep}
Now we define BGG operators and a variant of Demazure operators 
to discuss the representations of Hecke algebras on their commutative subalgebras.
Define $\bgg_\alpha:S_\Z(X)\to S_\Z(X)$ by the following formula:
\begin{align*}
\bgg_\alpha(f) = \frac{f-s_\alpha(f)}{\alpha}.
\end{align*}
As usual, the right side of the formula actually lies in $S_\Z(X)$.
Note that the commutativity relation for $\dhecke$ may be written
\begin{align*}
s_\alpha\cdot f - s_\alpha(f)\cdot s_\alpha = c_\alpha \Delta_\alpha(f),
\end{align*}
for any $f\in \dA$.
Define $\demazure_\alpha:\aA\to \aA$ by the following formula:
\begin{align*}
\demazure_\alpha(f) = \frac{f-s_\alpha(f)}{1-e^{-\alpha}}.
\end{align*}
Note that the commutativity relation for $\ahecke$ may be written
\begin{align*}
T_\alpha f - s_\alpha(f) T_\alpha = (q_\alpha-1)\demazure_\alpha(f),
\end{align*}
for any $f\in \aA$

Recall that the algebra $\idA$ is a subalgebra of the localization $\aA\otimes_\Z \Z[h^{\pm 1}]$
by the inclusion which maps $P_x$ to $h^{-1}(e^x-1)$. 
We compute:
\begin{align*}
h \demazure_\alpha(P_x) &= \demazure_\alpha(e^x-1) \\
&= \demazure_\alpha(e^x) \\
&= \begin{cases}
0 
& \text{if $\langle x,\check{\alpha}\rangle = 0$}, \\
\langle x,\check{\alpha}\rangle + h
(P_x+ P_{x-\alpha}+\dots P_{s_\alpha(x)+\alpha})  
& \text{if $\langle x,\check{\alpha}\rangle > 0$}, \\
\langle x,\check{\alpha}\rangle - h
(P_{x+\alpha}+ P_{x+2\alpha}+\dots P_{s_\alpha(x)}) 
& \text{if $\langle x,\check{\alpha}\rangle < 0$},
\end{cases}
\end{align*}
and see that $\demazure_\alpha:\idA\to h^{-1}\idA$. 
To put it informally, $\demazure_\alpha$ is singular at $h=0$. 
Nonetheless we have a well defined operator $h\demazure_\alpha:\idA\to \idA$. 
This is summarized by the following.
\begin{claim}
Let $f\in \idA$ and consider the operator $h\demazure_\alpha:\idA\to \idA$. 
The commutativity relation for the interpolating Hecke algebra may be written,
\begin{align*}
T_\alpha f - s_\alpha(f) T_\alpha = c_\alpha h\demazure_\alpha(f).
\end{align*}
In the specialization $\iA\otimes_{\Z[h]}\Z[h]/(h)\cong S(X)$, the operator $h\demazure_\alpha$ specializes to $\bgg_\alpha$.
\end{claim}

We also remark that the classical Demazure operators, $\widetilde{\demazure}_\alpha:\aA\to\aA$, given by
\begin{align*}
\widetilde{\demazure}_\alpha:f\mapsto \dfrac{f-e^{-\alpha}s_\alpha(f)}{1-e^{-\alpha}},
\end{align*}
may be expressed in terms of $\demazure_\alpha$.
Let $\rho\in X$ be defined by $\langle \rho,\check{\alpha}\rangle =1, \alpha\in \Pi$, and recall that $\widehat{e^\rho}:\aA\to\aA$ is the invertible operator given by multiplication by $e^\rho$. 
We claim that,
\begin{align*}
\widetilde{\demazure}_\alpha = \widehat{e^{-\rho}}\circ \demazure_\alpha\circ \widehat{e^\rho}.
\end{align*}
Indeed, $s_\alpha(e^\rho) = e^{-\alpha}e^\rho$, hence,
\begin{align*}
\widehat{e^{-\rho}}\circ \demazure_\alpha\circ \widehat{e^\rho}(f) 
=& e^{-\rho} \cdot \dfrac{e^\rho f - s_\alpha(e^\rho f)}{1-e^{-\alpha}}, \\
=& \dfrac{f-e^{-\alpha}s_\alpha(f)}{1-e^{-\alpha}}.
\end{align*}
As the classical Demazure operators satisfy the braid relations, so do the $\demazure_\alpha$ 
\begin{align*}
\cdots \demazure_\beta \demazure_\alpha = 
\cdots \demazure_\alpha \demazure_\beta, \hspace{10pt} \text{$m_{\alpha,\beta}$ terms}.
\end{align*}
It is also important to note that $\demazure_\alpha$ does not specialize to $\bgg_\alpha$ when
$h\to 0$, but rather $h\demazure_\alpha$ specializes to $\bgg_\alpha$. 
This proves that the $\bgg_\alpha$ also satisfy the braid relations.
Further, the quadratic relation $\demazure_\alpha^2 = \demazure_\alpha$ gives $(h\demazure_\alpha)^2 = h(h\demazure_\alpha)$, which specializes to $0$ when $h\to 0$, showing $\bgg_\alpha^2=0$. 

We remark that for $w\in W$, the operator $\demazure_w$ is uniquely defined by taking a reduced decomposition $w=s_1\cdots s_r$ and setting $\demazure_w = \demazure_{s_1}\cdots \demazure_{s_r}$.
The assumption that $(X,Y,R,\check{R},\Pi)$ is simply connected implies by the Pittie-Steinberg theorem that $W$ forms a basis of $\End_{(\idA)^W}(\idA)$ as a left $\idA$-module.
Extending scalers from $\idA$ to the fraction field $ff(\idA)$, we see by a triangular base change that $\{\demazure_w\mid w\in W\}$ forms a basis of $\End_{ff(\idA)^W}(ff(\idA))$.

\begin{proposition}
Consider the representation of the finite Hecke algebra $\fhecke$ on $\pC[h]$ given
by sending $T_\alpha\mapsto q_\alpha$. 
There is an induced representation of $\idhecke$ on $\idhecke\otimes_{\fhecke}\pC[h]\cong \idA\otimes_{\Z[h]}\pC[h]$.
Write $\widehat{T_\alpha},\widehat{q_\alpha},\dots : \idA\otimes_{\Z[h]}\pC[h]\to\idA\otimes_{\Z[h]}\pC[h]$ for the action of $T_\alpha,q_\alpha,\dots$ as operators on $\idA\otimes_{\Z[h]}\pC[h]$. 
Then,
\begin{align}\label{t-q}
\widehat{T_\alpha-q_\alpha}:P_x\mapsto 
(c_\alpha+q_\alpha P_{-\alpha})\cdot h\demazure_\alpha(P_x).
\end{align}
\end{proposition}
\begin{proof}
The claim is obvious for $\langle x, \check{\alpha}\rangle =0$. We show the case 
$\langle x,\check{\alpha}\rangle > 0 $, the other case being nearly identical.
By the commutativity relation for $\idhecke$ we find:
\begin{align*}
\widehat{T_\alpha -q_\alpha }(P_x) 
=& P_{s_\alpha(x)}T_\alpha\otimes 1 +(q_\alpha-1)\demazure_\alpha(P_x) - q_\alpha P_x,\\
=& q_\alpha P_{s_\alpha(x)} - q_\alpha P_x + c_{s_\alpha}h \demazure_\alpha(P_x)\\
=& q_\alpha (P_{s_\alpha(x)} - P_{x}) + (q_\alpha -1)
(h^{-1}\langle x,\check{\alpha}\rangle +P_x+\cdots +P_{s_\alpha(x)+\alpha}) \\
=& q_\alpha 
\left( h^{-1}\langle x,\check{\alpha}\rangle +P_{x-\alpha}+\cdots +P_{s_\alpha(x)}\right) - 
\left( h^{-1}\langle x,\check{\alpha}\rangle +P_x+\cdots +P_{s_\alpha(x)+\alpha}  \right) \\
=& (q_\alpha (1+hP_{-\alpha})-1)\demazure_\alpha(P_x) \\
=& (q_\alpha-1+q_\alpha hP_{-\alpha})\demazure_\alpha(P_x) \\
=& (c_{s_\alpha}+q_\alpha P_{-\alpha})h\demazure_\alpha(P_x) \\
=& (c_{s_\alpha}+P_{-\alpha}+hc_{s_\alpha}P_{-\alpha})h\demazure_\alpha(P_x).
\end{align*}
Here we have used that $\demazure_\alpha(P_x) = h^{-1}\demazure_\alpha(U_x)$, and that
$1+hP_{-\alpha }=U_{-\alpha}$ which satisfies the relation,
$U_{-\alpha}U_y = U_{y-\alpha}$.
\end{proof}

\subsection{Weight spaces of $\ihecke$-modules}
For this section fix $k$ a field and fix a parameter $h_0\in k$.
\begin{definition}
Let $\iA = \idA\otimes_{\Z[h]}k[h] / (h-h_0)$, $\iT = \Hom_{k-alg}(\iA,k)$.
We call $\iT$ the space of weights for the algebra $\iA$, and given $\lambda\in \iT$ and $f\in \iA$ we let $f(\lambda)\in k$ denote the evaluation of $\lambda$ at $f$.
\end{definition}
As $\iA\otimes_{\Z[h]}k[h]/(h-h_0)$ is isomorphic to the symmetric algebra of $X$ over $k$ for $h_0=0$ and the group ring of $X$ over $k$ for $h_0\not=0$, we find that $\iT\cong Y\otimes_\Z k$ for $h_0=0$ and $\iT\cong Y\otimes_\Z k^*$ for $h_0\not=0$.
Motivated by the following paragraph we will call $\iT$ the 
space of weights for the algebra $\iA$ with parameter $h_0$.

Let $V$ be an $\iA$-module which is finite dimensional as a $k$-vector space. 
If $k$ is algebraically closed, there is 
canonical generalized eigenspace (weight space) decomposition:
\begin{align*}
V = \bigoplus_{\lambda\in \Omega} V_\lambda,
\end{align*}
where for $\lambda\in \iT$, 
$V_\lambda = \{v\in V\mid (f-\lambda(f))^n v=0$ for all $f\in \iA, n\gg 0\}$ 
and $\Omega = \Omega(V) \subset \iT$ is the finite subset of $\lambda$ such that
$V_\lambda\not=0$.

In discussing the weight spaces, $V_\lambda$, it is convenient to introduce 
a non-unital localization $\lA$ of $\iA$. We set
\begin{align*}
\lA = \bigoplus_{\lambda\in \iT} \iA_\lambda,
\end{align*}
where $\iA_\lambda = \iA[f^{-1}\mid \lambda(f)\not=0]$ with
the unit element denoted by $1_\lambda$. 
We have an equivalence from the category of finite dimensional
$\iA$-modules with eigenvectors in $k$ to the category of unital (with $1_\lambda$ as the projection onto $V_\lambda$) finite dimensional 
$\lA$-modules, sending $V\mapsto \bigoplus_{\lambda\in \Omega}V_\lambda$.

Now fix some set of parameters for $\fhecke$ in $k$, in other words, fix an algebra morphism $\pC[h]\to k$ for which $h\to h_0$. 
Define $\ihecke = \idhecke\otimes_\pC k$.
As we will see, if $V$ is a finite dimensional representation of $\ihecke$
and $\lambda\in\Omega(V)$ is a weight of the subalgebra $\iA\subset \ihecke$ 
which is not invariant under 
$s_\alpha$, then $T_\alpha $ does not
preserve the weight space $V_\lambda$, nor does it permute the weight spaces.
In fact $T_\alpha  (V_\lambda)\subset V_\lambda\oplus V_{s_\alpha(\lambda)}$.
There is, however, a relation
$1_\lambda T_\alpha  1_\lambda = f_\alpha1_\lambda$, in
$\End_\C(V)$, where $f_\alpha\in \iA_\lambda$ will be made explicit. 
In this work we describe generators and relations of $\ihecke$ 
which permute the weight spaces of finite representations.
These generators are based on the elements $1_{s_\alpha(\lambda)}T_\alpha 1_\lambda$.

Before we define the localized Hecke algebra $\lhecke$, note that $(-P_{-\alpha})\cdot h\demazure_\alpha(f) = f-s_\alpha(f)$.
This relation may be used to extend $h\demazure_\alpha$ to an operator on $\lA$ as follows.
Let $f\in \iA_\lambda$ be given. If $s_\alpha(\lambda)\not=\lambda$, then $P_{-\alpha}$ is invertible in both $\iA_\lambda$ and $\iA_{s_\alpha(\lambda)}$, and we set,
\begin{align*}
h\demazure_\alpha(f) = (-P_{-\alpha})^{-1}f - (-P_{-\alpha})^{-1}s_\alpha(f)
\in \iA_\lambda\oplus \iA_{s_\alpha(\lambda)}.
\end{align*}
If $s_\alpha(\lambda)=\lambda$, we appeal to the description of $\iA$ as one of $\aA$ or $\dA$, where the action of $h\demazure_\alpha(f)$ may be written as a fraction lying in $\iA_\lambda$.
This gives a well defined operator $h\demazure_\alpha: \lA\to \lA$.

\begin{definition}
As an additive group let
$\lhecke = \ihecke\otimes_{\iA}\lA
\cong \bigoplus_{\lambda\in \iT} \fhecke \otimes_\C \iA_\lambda$.
Let the algebra $\lA$ be a subalgebra of $\lhecke$. 
For $\Lambda\in \iT/W$, put $1_\Lambda = \sum_{\lambda\in\Lambda}1_\lambda$ and 
$T_\alpha ^\Lambda = T_\alpha \otimes 1_\Lambda\in \lhecke$. 
For each such $\Lambda$ let the inclusion of the finite Hecke algebra,
\begin{align*}
\fhecke&\hookrightarrow \bigoplus_{\lambda\in \Lambda} 
\ihecke\otimes_{\iA}\iA_\lambda \subset \lhecke, \\
T_\alpha &\mapsto T_\alpha ^\Lambda.
\end{align*}
be an algebra morphism.
The following decomposition,
\begin{align*}
\lhecke = \bigoplus_{\Lambda\in\iT/W} \lhecke 1_\Lambda.
\end{align*}
gives $\lhecke$ the structure of a locally unital algebra with the following commutativity relation:
\begin{align*}
T_\alpha^\Lambda f - s_\alpha(f)T_\alpha^\Lambda 
= c_\alpha h\demazure_\alpha(f),
\end{align*}
for any $f\in \lA 1_\Lambda$. 

\end{definition}

First note that the unital algebra $\lhecke 1_\Lambda$ in the direct sum above is a subalgebra of the completion of $\ihecke$ at the kernel of the associated central character $\Lambda: (\iA)^W\to k$ \cite{lusztig}. 
It is in fact the subalgebra generated by $\ihecke$ and the localizations $\iA_\lambda 1_\lambda$ in the completion of $\iA$ with respect to this ideal.

We show that a finite dimensional representation of $\ihecke$ for which the eigenvectors for $\iA$ are in $k$ indeed gives rise to a 
representation of $\lhecke$, where $1_\lambda$ acts as the projection onto
the $\lambda$ weight space.

\begin{lemma}\label{commutativity-relation}
Let $V$ be a finite dimensional representation of $\ihecke$ for which the eigenvalues of $\iA$ are in $k$. Let 
$1_\lambda\in \End_\C(V)$ be the projection onto $V_\lambda$. If
$\alpha\in \Pi$ with $s_\alpha(\lambda)=\lambda$ then 
$T_\alpha (\lambda)\subset V_\lambda$, hence,
\begin{align*}
T_\alpha 1_\lambda =& 1_\lambda T_\alpha \\
=& 1_\lambda T_\alpha + c_\alpha h\demazure_\alpha(1_\lambda).
\end{align*}

Moreover, if $s_\alpha(\lambda)\not=\lambda$, then $(P_{-\alpha})^{-1}\in \iA_\lambda,\iA_{s_\alpha(\lambda)}$. Thus the eigenvalues of $P_{-\alpha}$ are non-zero and hence $P_{-\alpha}$ is an invertible operator on $V_\lambda, V_{s_\alpha(\lambda)}$.
Also, $T_\alpha(V_\lambda)\subset V_\lambda\oplus V_{s_\alpha(\lambda)}$
and the following commutativity relation holds,
\begin{align*}
T_\alpha 1_\lambda - 1_{s_\alpha(\lambda)}T_\alpha 
=& 
c_\alpha(-P_{-\alpha})^{-1}(1_\lambda - 1_{s_\alpha}), \\
=&
c_\alpha h\demazure_\alpha(1_\lambda).
\end{align*}

\end{lemma}
\begin{proof}
Let $x\in X$ and $z\in k$.
With the simple identity, $a^N-b^N = (a-b)\sum_{0\le i\le N-1}a^ib^{N-i-1}$ we find
\begin{align*}
(-P_{-\alpha})h\demazure_\alpha((P_x-z)^N) 
=& 
(P_x-z)^N - (P_{s_\alpha(x)}-z)^N, \\
=&
(P_x-P_{s_\alpha(x)})
\sum_{0\le i\le N-1} (P_x-z)^i(P_{s_\alpha(x)}-z)^{N-i-1},\\
=&
\left((-P_{-\alpha})h\demazure_\alpha(P_x)\right) \cdot 
\sum_{0\le i\le N-1} (P_x-z)^i(P_{s_\alpha(x)}-z)^{N-i-1}.
\end{align*}
As $\iA$ is a domain,
\begin{align*}
(P_x-z)^N T_\alpha =& T_\alpha (P_{s_\alpha(x)}-z)^N + 
h\demazure_\alpha(P_x)
\sum_{0\le i\le N-1} (P_x-z)^i(P_{s_\alpha(x)}-z)^{N-i-1}.
\end{align*}

Let $\lambda\in\iT$ with $s_\alpha(\lambda)=\lambda$ and suppose $N, z$ are such that
$(P_{x}-z)^{\lfloor N/2\rfloor} (V_\lambda)=(P_{s_\alpha(x)}-z)^{\lfloor N/2\rfloor}(V_\lambda)=0$.
In this case, the above expression shows that
$(P_{x}-z)^N T_\alpha1_\lambda v=0$ for $v\in V$. Thus, $T_\alpha(V_\lambda)\subset V_\lambda$ and $T_\alpha1_\lambda = 1_\lambda T_\alpha$.

Now, suppose
$s_\alpha(\lambda)\not=\lambda$ so that $\lambda(P_{-\alpha})\not=0$.
Thus, $(P_{-\alpha})^{-1}\in \iA_\lambda, \iA_{s_\alpha(\lambda)}$ and $(P_{-\alpha})^{-1}$ may be considered as an operator on $V_\lambda, V_{s_\alpha(\lambda)}$ (as the operator $P_{-\alpha}$ has a lone eigenvalue which is non-zero). We also suppose that $N,z$ are picked so that
$(P_{x}-z)^N(V_\lambda)=0$. Then,
\begin{align*}
(P_{s_\alpha(x)}-z)^N (T_\alpha 1_\lambda-c_\alpha(-P_{-\alpha})^{-1}1_\lambda) v =& 
T_\alpha (P_{s_\alpha(x)}-z)^N 1_\lambda v \\ 
&+c_\alpha(-P_{-\alpha})^{-1} (P_{s_\alpha(x)}-z)^N1_\lambda v, \\
=& 0.
\end{align*}
In particular, it follows that $T_\alpha(V_\lambda)\subset V_\lambda\oplus V_{s_\alpha(\lambda)}$,
and moreover,
\begin{align*}
T_\alpha1_\lambda -  1_{s_\alpha(\lambda)}T_\alpha  
=& 
c_\alpha(-P_{-\alpha})^{-1}(1_\lambda - 1_{s_\alpha}), \\
=&
c_\alpha h\demazure_\alpha(1_\lambda).
\end{align*}
in $\End_\C(V)$.

\end{proof}

To show that a finite dimensional representation of $\ihecke$ lifts to a finite dimensional
representation of $\lhecke$ we note that the ring of $W$-invariants, $(\iA)^W$, is the center of $\ihecke$
and each finite representation splits into a direct sum of generalized eigenspaces
$V_\Lambda$ of the center of $\ihecke$ according to the central characters 
$\Lambda\in \iT/W$:
\begin{align*}
V = \sum_{\Lambda\in \iT/W} V_\Lambda.
\end{align*}
We can decompose the operators $T_\alpha  = \sum_{\Lambda}T_\alpha ^\Lambda$,
$f = \sum_\lambda f_\lambda$. The above lemma shows that these operators give
an action of $\lhecke$.

\section{Localized Quiver Hecke Algebras}\label{generators-relations}

\subsection{Weyl group orbits on tori}
We will need the following lemmas to study the localized quiver Hecke algebras.
Let $W$ be the Weyl group of the reduced root datum, $(X, Y, R,\check{R}, \Pi)$.
Recall that $\Pi\subset R$ defines a length function $\ell$ on $W$.

\begin{definition}
Let $W'\subset W$ be a subgroup.
We say that $W'$ is a standard parabolic subgroup if there is a subset $\Pi^\lambda\subset \Pi$
so that $W'$ is the subgroup generated by $\{s_{\alpha}\in W\mid \alpha\in \Pi^\lambda\}$. 
We call a subgroup parabolic if it is $W$-conjugate to a standard parabolic subgroup.

Let $\lambda\in \iT $.
We say that $\lambda$ is standard parabolic, if there is a subset $\Pi^\lambda\subset \Pi$
so that the stabilizer of $\lambda$
in the Weyl group $W$ is the standard parabolic subgroup generated by 
$\{s_{\alpha}\in W\mid \alpha\in \Pi^\lambda\}$. 
We call a weight parabolic (resp. parabolic with respect to $W^P$) 
if it is in the $W$-orbit (resp. $W^P$-orbit) of a standard parabolic weight (resp. standard parabolic with respect to $W^P$).
\end{definition}

\begin{lemma}\label{permute}
Let $w = s_{\alpha_n}\cdots s_{\alpha_1} = s_{\beta_n}\cdots s_{\beta_1}$ be two reduced
expressions. Then there is a permutation $p$ of $\{ 1,\cdots ,n\}$ so that
whenever $p(i)=j$,
\begin{align*}
s_{\alpha_1}\cdots s_{\alpha_{i-1}}(\alpha_i) &= 
s_{\beta_1}\cdots s_{\beta_{j-1}}(\beta_{j}), \\
s_{\alpha_n}\cdots s_{\alpha_{i+1}}(\alpha_i) &=
s_{\beta_n}\cdots s_{\beta_{j+1}} (\beta_j).
\end{align*}
\end{lemma}
\begin{proof}
This is simply a restatement of the following standard theorem, see \cite[Section 1,7]{humphreys}.

Let $w = s_{\alpha_1}\cdots s_{\alpha_n}$ be a reduced decomposition. Put 
$\gamma_i = s_{\alpha_{n}}\cdots s_{\alpha_{i+1}}(\alpha_i)$. Then the roots 
$\gamma_1, \ldots , \gamma_n$ are all distinct and the set 
$\{\gamma_1, \ldots , \gamma_n\}$ equals $R^+\cap w^{-1}R^-$, which is the set of 
$\gamma\in R^+$ such that $w(\gamma)\in R^-$.

\end{proof}

Let $(X,Y,R,\check{R},\Pi)$ be a root data associated to the $k$-split semisimple algebraic group $G$. 
Associated to this data is a maximal torus, $\aT  = Y\otimes_{\Z} k^*$ and a dual torus, $\aT ^* = X\otimes_\Z k^*$, both with actions of the Weyl group $W$. 
The purpose of this section is to collect some results on the stabilizers of elements of $\aT $ in the Weyl group $W$.

Let $\lambda\in \aT$ and denote by $\langle \lambda \rangle$ the smallest closed subgroup of $\aT$ containing $\lambda$. 
\begin{claim}
The subgroup $\langle \lambda \rangle$ is the direct sum of a cyclic subgroup generated by $\zeta\in\aT$ with finite order and a torus $\aS\subset \aT$.
\end{claim}
\begin{proof}
The subgroup $\langle \lambda \rangle$ has finitely many components.
Let $\aS$ be the identity component of $\langle \lambda \rangle$.
The morphism from $\langle \lambda \rangle\to \langle \lambda \rangle / \aS$ is a split surjection, as the category of diagonalizable groups is dual to the category of finitely generated abelian groups.
Considering that the set of components of $\langle \lambda \rangle$ containing the powers $\lambda^i, i\in \Z$ is a closed subgroup of $\aT$, we find that the group of components is a cyclic group.
It follows that $\langle \lambda \rangle$ is the direct sum of a cyclic subgroup generated by $\zeta\in\aT$ with finite order and a torus $\aS\subset \aT$.
\end{proof}

\begin{definition}
Let $\tilde{\alpha}\in R^+$ be the highest root. 
Define an augmented standard parabolic subgroup of $W$ to be a subgroup $W'$ such that there is a subset $I\subset \Pi$ for which $W'$ is generated by $\{s_\alpha\}_{\alpha\in I}\cup \{s_{\tilde{\alpha}}\}$. 
A subgroup is called an augmented parabolic subgroup if it is $W$-conjugate to an augmented standard parabolic subgroup.
\end{definition}

\begin{corollary}
Assume $(X,Y,R,\check{R},\Pi)$ to be simply connected semisimple. 
Given $\lambda\in \aT$, the centralizer of $\lambda$ in the Weyl group $W$ is the intersection of a parabolic subgroup with an augmented parabolic subgroup.
\end{corollary}
\begin{proof}
As the Weyl group acts by algebraic automorphisms of $\aT$, the stabilizer of $\lambda$ in $W$ is equal to the intersection of the stabilizer in $W$ of $\aS$ with the stabilizer of $\zeta$. 
As $\aS\subset G$ is a torus, the centralizer $Z_G(\aS)$ is a connected, reductive subgroup \cite[section 26.2]{humphreys-lag}.
Moreover the centralizer $Z_G(\aS)$ is the Levi subgroup of a parabolic subgroup of $G$ (see \cite[Proposition 1.22]{digne}) and hence the centralizer $Z_W(\aS)$ is a parabolic subgroup.

It remains only to show that the centralizer of an element $\zeta\in \aT$ of finite order is an augmented parabolic subgroup. 
We may restrict our attention to the group $Y\otimes_\Z \langle \zeta \rangle$. 
If $\zeta$ has order $n$, then this is isomorphic to the group $Y\otimes_\Z \frac{1}{n}\Z/\Z$.
Lift $\zeta\in Y\otimes_\Z \langle \zeta \rangle$ to an element in the Euclidean space $z\in Y\otimes_\Z \frac{1}{n}\Z\subset Y\otimes_\Z \R$, and note that for $w\in W$ we have $w(\zeta)=\zeta$ if and only if $w(z)-z\in Y\otimes \Z\subset Y\otimes_\Z \R$.
This is the same as asserting that there exists $t\in Y$ with $w(z)-y = z$.
For $w\in W, y\in Y$ fixed, the transformation $x\mapsto w(x)-y$ is an element of the affine group $W_a$, which is the semidirect product of the Weyl group with the translation group $Y$.
As $Y$ has a basis given by $\{\check{\alpha}\}_{\alpha\in \Pi}$, $W_a$ is a Coxeter group with Coxeter generators $\{s_\alpha\}_{\alpha\in \Pi}\cup \{s_{\tilde{\alpha},1}\}$ where $s_{\tilde{\alpha},1}(x) = s_{\tilde{\alpha}}(x) + \check{\alpha}$ (see \cite{humphreys}).
Examining the Weyl group stabilizers of elements of the fundamental alcove $\{x\in Y\otimes_\Z \R\mid \langle x,\alpha\rangle>0 \alpha\in \Pi, \langle x,\tilde{\alpha}\rangle =1\}$, whose walls are the hyperplanes orthogonal to $\alpha, \alpha\in \Pi$ and $\tilde{\alpha}$, we find that the stabilizer of $\zeta\in \aT$ is conjugate to an augmented parabolic subgroup.
\end{proof}
\begin{remark}\label{split-projection}
Let $(X,Y,R, \check{R},\Pi)$ be any root data for which $X$ contains the fundamental weights $\{\omega_\alpha\}_{\alpha\in \Pi}\subset \Q\cdot R\subset X\otimes_\Z\Q$, and let $Y' = \Z \cdot \check{R}$ be the $\Z$-span of the coroots in $Y$.
We have a natural inclusion $i:Y'\hookrightarrow Y$.
Define the projection $p:Y\to Y'$ by the following formula,
\begin{align*}
p(y) = \sum_{\alpha\in \Pi}\omega_\alpha(y)\check{\alpha}.
\end{align*}
Then $p$ is a split surjection, $i\circ p(y') = y'$ for $y'\in Y'$.
Moreover, $p,i$ are both $W$-equivariant.
Abusing notation, let $p:Y\otimes_\Z k^*$ and $i:Y'\otimes_\Z k^*$ denote the corresponding maps on tori. 
Again, $p:Y\otimes_\Z k^*\to Y'\otimes_\Z k^*$ is a split $W$-equivariant surjection, and hence is a bijection on $W$-orbits.
\end{remark}
 
We compute the stabilizers of weights in rank 2, which by the above remark is equivalent to computing $W^{\alpha,\beta}$-orbits on general tori associated to simply connected root data.
Given a rank 2 root system let $(X,Y,R,\check{R},\Pi)$ be a root data for which $X$ has the basis $\{\omega_\alpha\}_{\alpha\in \Pi}$ and $Y$ has the basis $\{\check{\alpha}\}_{\alpha\in \Pi}$.
We compute the stabilizers of elements $\lambda\in \aT = Y\otimes_\Z k^*$.

In type $A_1\times A_1, A_2$ we find that the stabilizer of any element is already a parabolic subgroups of $W$.

In types $B_2, G_2$, consider the action of $W$ on the vector space $Y\otimes_\Z k$.
The only non-parabolic, augmented standard parabolic subgroup of $W$ is generated by the reflection $s_{\tilde{\alpha}}$ and one other reflection.
As the two reflections have distinct 1-dimensional eigenspaces we find that the subspace of fixed points on $Y\otimes_\Z k = \operatorname{Lie}(Y\otimes_\Z k^*)$ is the zero subspace.
It follows that the group of fixed points on $\aT$ of a non-parabolic, augmented standard parabolic subgroup of $W$ is a finite torsion subgroup of $\aT$.

Consider, as above, the real vector space $E = Y\otimes_\Z \R$.
Recall that the torsion subgroup of $\aT$ may be embedded as a subgroup of $Y\otimes_\Z \Q \subset E / Y$.
It is plain to check that the set of elements $\lambda$ in the fundamental alcove of $E$ which are stabilized by $s_{\tilde{\alpha},1}$ and one other Coxeter reflection actually lie in $\frac{1}{2} Y$, which gives us the following corollary.

\begin{corollary} \label{orbit}
If $(X,Y,R,\check{R},\Pi)$ is a simply connected semisimple root system of rank 2, and $\lambda\in \aT = Y\otimes_\Z k^*$ satisfies $\operatorname{stab}_{W}(\lambda)$ is a non-parabolic subgroup of $W$, then $\lambda$ satisfies $\lambda^2 = 1\in \aT$, and for $\alpha\in \Pi$,
\begin{align*}
\langle \alpha,\lambda\rangle = \begin{cases}
1 \text{, if $s_\alpha(\lambda) = \lambda$,} \\
-1 \text{, if $s_\alpha(\lambda) \not = \lambda$.}
\end{cases}
\end{align*}
\end{corollary}

\subsection{The datum $G$ and its conditions} \label{conditions}
\begin{remark}
We now define an abstract set of datum on which our definition of localized quiver Hecke algebra depends.
The definition of localized quiver Hecke algebra we give extends that of \cite{rouquier-2km}, \cite{rouquier-qha}.
We briefly paraphrase the set-up of \cite{rouquier-qha} where actually two quivers are given. 

Fix a finite set $I$ as well as a Cartan matrix $C = (a_{i,j})_{i,j\in I}$.
This determines the first quiver, whose vertices are identified with $I$, and with $-a_{i,j}$ arrows from $i$ to $j$.
Next, consider the action of the symmetric group $\mathfrak{S}_n$ on the set $I^n$.
We form the quiver $\Psi_{I,n}(\Lambda)$ whose vertices are identified with a fixed $\mathfrak{S}_n$-orbit $\Lambda$ in $I^n$.
There are two types of arrows for the quiver $\Psi_{I,n}(\Lambda)$. 
The first type is an arrow $x_i:\lambda\to \lambda$ for each $1\le i \le n$, 
and second type is an arrow $\tau_i:\lambda\to s_i(\lambda)$, where $s_i\in \mathfrak{S}_n$ is the element $(i,i+1)$.

Finally, the first quiver determines a set of polynomials $Q = (Q_{i,j}(u,v))_{i,j\in I}$ with which one constructs for each $\mathfrak{S}_n$-orbit in $I^n$ the \emph{quiver Hecke algebra}, $H(Q)_\Lambda$.
The algebra $H(Q)_\Lambda$ is a quiver algebra with relations on the second quiver $\Psi_{I,n}$, with relations determined by the polynomial data $Q$.
For $\Gamma$ of type $A$ the algebras $H(Q)_\Lambda$ may be identified with localizations of affine Hecke algebras of type $A$.
The quiver Hecke algebras defined in this paper do not give a way of generalizing the first quiver, but instead generalize the data $Q$ and second quiver $\Psi_{I,n}(\Lambda)$ to give a quiver Hecke algebra presentation of affine Hecke algebras of any type.

\end{remark}

Fix $h_0\in k$. 
Then $\qA$ is the specialization $\qA = \idA\otimes_{\Z[h]}k[h]/(h-h_0)$ of the interpolating ring $\idA$ at $h\to h_0$,
and $h\demazure_\alpha$ as the specialization of the Demazure operator. 
Thus, either $h_0=0$ in which case $\qA$ is isomorphic to the symmetric algebra and $h\demazure_\alpha$ is the BGG operator $\bgg_\alpha$, or $h_0\not=0$ in which case $\qA$ is isomorphic to the group ring $k[X]$ and $h\demazure_\alpha$ is a scaler multiple of the Demazure operator.
Recall that $\qT = \Hom_{k-alg}(\qA,k)$ is the set of $k$-algebra homomorphisms from $\qA$ to $k$, which is either isomorphic to $\mathfrak{h}=Y\otimes_\Z k$ for $h_0=0$, or isomorphic to $\aT=Y\otimes_\Z k^*$, for $h_0\not=0$. We have 
\begin{align*}
\lqA=\bigoplus_{\lambda\in \qT} \qA_\lambda,
\end{align*}
the localized, non-unital algebra associated to $\qA$, where $\qA_\lambda$ is the localization of $\iA$ at $\lambda\in \iT$.

Let $G=(G_\alpha^\lambda)_{\lambda\in \qT , \alpha\in \Pi}$
be a collection of non-zero rational functions, 
$G_{\alpha}^\lambda\in \qA_\lambda\backslash \{0\} = \lqA1_\lambda\backslash \{0\}$.
We list now a few conditions that this data is required to satisfy before we can define the localized quiver Hecke algebra $\qhecke$.

For notational purposes, we need the following lemma.
\begin{lemma}\label{root-choice}
Let $\alpha,\beta\in\Pi$. We put $W^{\alpha,\beta} = \langle s_\alpha,s_\beta\rangle$,
the subgroup of $W$ generated by $s_\alpha,s_\beta$. We also put
$m = m_{\alpha,\beta}$ as the order of $s_\alpha s_\beta$ when $\alpha\not=\beta$ and $2$ otherwise,
and finally we write $w_{\alpha,\beta}$ for the longest element in $W^{\alpha,\beta}$.
We have,
\begin{align*}
w_{\alpha,\beta}(\alpha) = 
\begin{cases}
-\alpha & \text{ if $m$ even},\\
-\beta & \text{ if $m$ odd}.
\end{cases}
\end{align*}
Instead of using cases, we will simply write $w_{\alpha,\beta}s_\alpha(\alpha)$
which is equal to $\alpha$ for $m_{\alpha,\beta}$ even, and $\beta$ for $m_{\alpha,\beta}$
odd.

\end{lemma}

\begin{definition}
Let $\lambda\in \qT $ $\alpha\in \Pi$. The weight $\lambda$ is said to be $\alpha$-exceptional
if there exists $\beta\in \Pi$ with $\lambda$ not parabolic with
respect to $W^{\alpha,\beta}$ and $s_\alpha(\lambda)\not = \lambda$.
If $\beta$ is as above, then it is automatic that $m_{\alpha,\beta} = 4,6$.
\end{definition}

Let $G = (G_\alpha^\lambda)_{\lambda\in \iT, \alpha\in \Pi}$ and
assume that $G$ satisfies the following conditions.
\begin{enumerate}
\item
For $\lambda,\alpha$ as above,
\begin{align*}
s_\alpha(G_\alpha^\lambda) = 
G_\alpha^{s_\alpha(\lambda)}.
\end{align*}
We shall refer to this as the \emph{associative relation} on $G$.
\item
If $s_\alpha(\lambda)=\lambda$ then $G_\alpha=1$.
\item
For any $\alpha,\beta\in \Pi$, 
\begin{align*}
w_{\alpha,\beta} s_\alpha (G_{\alpha}^\lambda) = 
G_{w_{\alpha,\beta}s_\alpha(\alpha)}^{w_{\alpha,\beta} s_\alpha(\lambda)}
\end{align*}
where, again, $w_{\alpha,\beta}\in W^{\alpha,\beta}$ is the longest element. We will refer to this relation
as the \emph{braid relation} on $G$.
Note that in the case $\alpha=\beta,$ we have that $w_{\alpha,\alpha} = s_\alpha$ 
and the condition is vacuous. 
\item
If $\lambda$ is $\alpha$-exceptional then $G_\alpha^\lambda=1$.
\end{enumerate}

\subsection{The localized quiver Hecke algebra} \label{relations}
Let 
$G = (G_\alpha^\lambda)_{\lambda\in \qT , \alpha\in \Pi}$
be a collection satisfying the conditions in section \ref{conditions}.
We define the localized quiver Hecke algebra $\qhecke$ associated to this choice in analogy
with a quiver algebra with relations over the ring $\aA$, see \cite{rouquier-qha}. 
Underlying this construction of a quiver algebra with relations is the quiver with vertices $\qT $,
and arrows $f:\lambda\to \lambda$,
$r^\lambda_{\alpha}:\lambda\to s_{\alpha}(\lambda)$, whenever 
$f\in \qA_\lambda,  \alpha\in \Pi$. We remark that the arrows $r_\alpha^\lambda$ give precisely the Cayley graph of the action of $\{s_\alpha\mid \alpha\in \Pi\}$ on $\iT$.

\begin{definition}
First, define 
$\widetilde{\qhecke}$ as the non-unitary algebra given by adjoining
generators $r_{\alpha}^\lambda$ to $\lqA$ which satisfy the following relations.
\begin{itemize}
\item
$r_{\alpha}^{\lambda} 1_{\nu} = 1_{s_{\alpha}(\nu)}r_{\alpha}^{\lambda} =
\delta_{\lambda, \nu}r_{\alpha}^{\lambda}$.
\item
$r_{\alpha}^{s_{\alpha}(\lambda)}r_{\alpha}^{\lambda} = 
\begin{cases}
G_\alpha^\lambda
&\text{ if $s_\alpha(\lambda) \not = \lambda$,} \\
h_0 r_\alpha^\lambda
&\text{ if $s_\alpha(\lambda)=\lambda$}.
\end{cases}$
\item
For $f\in \aA_\lambda$, 
\begin{align*}
r_{\alpha}^{\lambda} f - s_{\alpha}(f)r_{\alpha}^{\lambda}
=\begin{cases} 0 &\text{, if $s_\alpha(\lambda)\not =\lambda$,}\\
h \demazure_\alpha(f) &\text{, if $s_\alpha(\lambda)=\lambda$}.
\end{cases}
\end{align*}
\item
For $\alpha, \beta\in \Pi$ distinct and $\lambda\in \qT $ a \emph{standard parabolic weight}
with respect to $\{\alpha, \beta\}$, 
or $\lambda$ which is not parabolic with respect to $W^{\alpha,\beta}$, but fixed by one of $s_\alpha, s_\beta$:
\begin{align*}
r_{\alpha_m}^{\lambda_m}\cdots r_{\alpha_1}^{\lambda_1} = 
r_{\beta_m}^{\mu_m}\cdots r_{\beta_1}^{\mu_1},
\end{align*}
where $m=m_{\alpha, \beta}$ is the order of $s_{\alpha}s_\beta$ in $W$, 
\begin{align*}
\alpha_i =& \begin{cases}
\alpha &\text{, $i$ odd,} \\
\beta & \text{, $i$ even,}
\end{cases} \\
\beta_i =& \begin{cases}
\beta &\text{, $i$ odd,} \\
\alpha & \text{, $i$ even,}
\end{cases} \\
\end{align*}
and $\lambda_i = s_{\alpha_{i-1}}\cdots s_{\alpha_1}(\lambda)$, 
$\mu_i = s_{\beta_{i-1}}\cdots s_{\beta_1}(\mu)$. 
This is known as the braid relation.
\end{itemize}
Finally, let $\mathscr{I}_\lambda$ be the right $\qA_\lambda$-module
consisting of elements $I\in \widetilde{\qhecke}1_\lambda$ such that there is
$f\in \qA_\lambda \backslash \{0\}$ with $I\cdot f = 0$.
We define $\qhecke = \widetilde{\qhecke}/\bigoplus_\lambda I_\lambda$.
Thus, there is no right polynomial torsion in $\qhecke$. 
\end{definition}

Given simply connected root datum 
we will produce in section \ref{isomorphism} a family $G$ and
an isomorphism 
$\qhecke\to \lhecke$.

\subsection{The PBW property and a faithful representation} \label{PBW-section}
This section analyzes the structure of $\qhecke$.
We start by defining a filtration $(\mathscr{F}^n)$ on $\qhecke$, letting 
$\mathscr{F}^n\subset \qhecke$ be the right $\lqA$-linear span of 
all products $r_{\alpha_k}^{\lambda_k} \cdots r_{\alpha_1}^{\lambda_1}$
with up to $n$ terms in them. We see 
$\mathscr{F}^n\cdot \mathscr{F}^m\subset \mathscr{F}^{n+m}$.

\begin{lemma}
Fix some $\lambda\in \qT$ and let $\mf{B}_1 = ( \alpha_n, \cdots , \alpha_1)$,
$\mf{B}_2 = ( \beta_n, \cdots, \beta_1)$, $\alpha_i,\beta_i\in \Pi$, be two ordered collections
of simple roots with the same cardinality such that 
$s_{\alpha_n}\cdots s_{\alpha_1} = s_{\beta_n}\cdots s_{\beta_1}$.
Then,
\begin{align*}
r_{\alpha_n}^{\lambda_n} \cdots r_{\alpha_1}^{\lambda_1} - 
r_{\beta_n}^{\mu_n} \cdots r_{\beta_1}^{\mu_1}\in \mathscr{F}^{n-1},
\end{align*}
where, $\lambda_i = s_{\alpha_{i-1}}\cdots s_{\alpha_1}(\lambda)$,
$\mu_i = s_{\beta_{i-1}}\cdots s_{\beta_1}(\lambda)$.
\end{lemma}
\begin{proof}
We prove the assertion by induction on $n$. The cases of $n=0,1$ are trivial.

First, put $w=s_{\alpha_n}\cdots s_{\alpha_1}$. 
Suppose that $\ell(w)<n$.
By the deletion condition, there exists, $1\le i < j \le n$ with,
\begin{align*}
s_{\alpha_j}\cdots s_{\alpha_{i+1}} = s_{\alpha_{j-1}}\cdots s_{\alpha_{i}}.
\end{align*}
By induction, we may assume
\begin{align*}
r_{\alpha_j}^{\lambda_j}\cdots r^{\lambda_{i+1}}_{\alpha_{i+1}} - 
r^{\lambda_j'}_{\alpha_{j-1}}\cdots r^{\lambda'_{i+1}}_{\alpha_{i}}\in \mathscr{F}^{j-i-1},
\end{align*}
with the appropriately chosen $\lambda_k'=\lambda_k, 1\le k\le i+1, j+1\le k\le n$. Thus,
\begin{align*}
r_{\alpha_n}^{\lambda_n}\cdots r^{\lambda_{j+1}}_{\alpha_{j+1}}
\left(r_{\alpha_j}^{\lambda_j}\cdots r^{\lambda_{i+1}}_{\alpha_{i+1}} \right)
r_{\alpha_i}^{\lambda_i}\cdots r^{\lambda_1}_{\alpha_1} 
- 
r_{\alpha_n}^{\lambda'_n}\cdots r^{\lambda'_{j+1}}_{\alpha_{j+1}}
\left(r_{\alpha_{j-1}}^{\lambda'_{j}}\cdots r^{\lambda'_{i+1}}_{\alpha_{i}} \right)
r_{\alpha_i}^{\lambda'_i}\cdots r^{\lambda'_1}_{\alpha_1}
\in \mathscr{F}^{n-1}.
\end{align*}
Because $r^{\lambda'_{i+1}}_{\alpha_{i}} 
r_{\alpha_i}^{\lambda'_i}\in \mathscr{F}^{1}$, the second term is in 
$\mathscr{F}^{n-1}$, hence 
$r_{\alpha_n}^{\lambda_n} \cdots r_{\alpha_1}^{\lambda_1}\in \mathscr{F}^{n-1}$. 
The claim now follows for non-reduced expressions.

\vspace{10mm}

Now we show that the assertion is true in the case of a braid relation.
Let $\alpha, \beta\in \Pi$ be distinct, and let $W^{\alpha,\beta}$ be the 
dihedral subgroup of $W$ generated by $s_\alpha, s_\beta$. Let $m = m_{\alpha,\beta}$
be the order of $s_\alpha s_\beta$, and set 
\begin{align*}
\alpha_i &= \begin{cases}
\alpha ,& \text{ if $i$ odd}, \\
\beta ,& \text{ if $i$ even.} \end{cases} \\
\beta_i &= \begin{cases} 
\beta ,& \text{ if $i$ odd}, \\
\alpha ,& \text{ if $i$ even.} \end{cases}
\end{align*}
We will show,
\begin{align}\label{diff}
r_{\alpha_m}^{\lambda_m} \cdots r_{\alpha_1}^{\lambda_1} - 
r_{\beta_m}^{\mu_m} \cdots r_{\beta_1}^{\mu_1}\in \mathscr{F}^{m-1}.
\end{align}

First, suppose $\lambda$ is not parabolic with respect to $W^{\alpha,\beta}$. 
By analyzing the four
simply connected semisimple groups of rank 2 in lemma \ref{orbit} we see the only such 
$\lambda$ have $m_{\alpha,\beta}=4,6$.
In the case $m_{\alpha,\beta} = 4$, we may assume that $\beta$ is longer than $\alpha$ and $\langle \lambda, \beta\rangle, \langle \lambda, \alpha\rangle =\pm 1$, as $\lambda$ has order two.
Checking the four elements satisfying that requirement in the torus of the associated simply connected semisimple group, the $W^{\alpha,\beta}$-orbit of $\lambda$ must be of order exactly two, with the two elements being $\pi = \check{\beta}\otimes -1$ and $s_\alpha(\pi) = (\check{\alpha}\otimes -1)\cdot (\check{\beta}\otimes -1)$.
We find $s_\beta(\pi) = \pi$ and $s_\beta(s_\alpha(\pi)) = s_\alpha(\pi)$.
We have the following picture of the Cayley graph of the orbit:
\begin{align*}
{}^{s_\beta} \circlearrowleft \pi \xleftrightarrow{s_\alpha} s_\alpha(\pi) 
\circlearrowright {}^{s_\beta}.
\end{align*}
Both of these weights are $\beta$-exceptional, hence the braid relation holds for both $\pi, s_\alpha(\pi)$, thus the difference in question is in fact zero.

In the case $G_2$ we also assume that $\beta$ is the longer root.
Again, $\langle \lambda, \beta\rangle, \langle \lambda, \alpha\rangle =\pm 1$.
In this case the order of the $W^{\alpha,\beta}$-orbit must be exactly three, with the three elements given by $\pi = \check{\alpha}\otimes -1$, $s_\beta(\pi) = (\check{\alpha}\otimes -1)\cdot (\check{\beta}\otimes -1)$, and $s_\alpha s_\beta(\pi) = \check{\beta}\otimes -1$.
We find that $\pi$ is $s_\alpha$-invariant, and $s_\alpha s_\beta (\pi)$ is $s_\beta$ invariant, so these two weights are exceptional.
We have the following picture of the Cayley graph of this orbit:
\begin{align*}
{}^{s_\alpha} \circlearrowleft \pi \xleftrightarrow{s_\beta} s_\beta(\pi) 
\xleftrightarrow{s_\alpha} s_\alpha s_\beta(\pi) \circlearrowright {}^{s_\beta}.
\end{align*}
We will show that the braid relation for $\pi, s_\alpha s_\beta(\pi)$ implies the braid relation for $s_\beta(\pi)$, which is the only weight not fixed by one of $s_\alpha, s_\beta$.
This will also demonstrate the technique used in the general case, that the braid relation for a standard parabolic weight implies a (different) braid relation for the other weights in its $W^{\alpha,\beta}$-orbit.
With this in mind, recall that we wish to calculate the difference,
\begin{align*}
r_\beta^{\pi} \cdots r_\beta^{s_\alpha s_\beta(\pi)} r_\alpha^{s_\beta(\pi)} - 
r_\alpha^{s_\alpha s_\beta(\pi)} \cdots r_\alpha^{\pi} r_\beta^{s_\beta(\pi)}.
\end{align*}
We simply multiply this difference on the right by $r_\alpha^{s_\alpha s_\beta(\pi)}$:
\begin{align*}
r_\beta^{\pi} r_\alpha^{\pi}\cdots r_\beta^{s_\alpha s_\beta(\pi)} r_\alpha^{s_\beta(\pi)}r_\alpha^{s_\alpha s_\beta(\pi)} - 
r_\alpha^{s_\alpha s_\beta(\pi)} r_\beta^{s_\alpha s_\beta(\pi)} \cdots r_\alpha^{\pi} r_\beta^{s_\beta(\pi)}r_\alpha^{s_\alpha s_\beta(\pi)}.
\end{align*}
The first term simplifies as, $r_\alpha^{s_\beta(\pi)} r_\alpha^{s_\alpha s_\beta(\pi)} = G_\alpha^{s_\alpha s_\beta(\pi)}$. 
As $s_\alpha s_\beta(\pi)$ is $\alpha$-exceptional we have $G_\alpha^{s_\alpha s_\beta(\pi)} = 1$.
The last six elements in the product in the second term may be substituted by the braid relation for the $\beta$-exceptional weight $s_\alpha s_\beta(\pi)$:
\begin{align*}
r_\alpha^{s_\alpha s_\beta(\pi)} 
r_\beta^{s_\alpha s_\beta(\pi)} \cdots r_\alpha^{\pi}              r_\beta^{s_\beta(\pi)}r_\alpha^{s_\alpha s_\beta(\pi)} = 
r_\alpha^{s_\alpha s_\beta(\pi)}
r_\alpha^{s_\beta(\pi)} \cdots
r_\alpha^{s_\alpha s_\beta(\pi)} r_\beta^{s_\alpha s_\beta(\pi)}
\end{align*}

Notice that the first two terms of the product on the right side of the equality simplify;
\begin{align*}
r_\alpha^{s_\alpha s_\beta(\pi)}r_\alpha^{s_\beta(\pi)} = 
G_\alpha^{s_\beta(\pi)}.
\end{align*}
Again, $s_\beta(\pi)$ is $\alpha$-exceptional so $G_\alpha^{s_\beta(\pi)} = 1$.
It follows that the difference in question is equal to:
\begin{align*}
& r_\beta^{\pi} r_\alpha^{\pi}\cdots r_\beta^{s_\alpha s_\beta(\pi)} r_\alpha^{s_\beta(\pi)}r_\alpha^{s_\alpha s_\beta(\pi)} - 
r_\alpha^{s_\alpha s_\beta(\pi)} r_\beta^{s_\alpha s_\beta(\pi)} \cdots r_\alpha^{\pi} r_\beta^{s_\beta(\pi)}r_\alpha^{s_\alpha s_\beta(\pi)} \\
&= r_\beta^{\pi} r_\alpha^{\pi}\cdots r_\beta^{s_\alpha s_\beta(\pi)}
G_\alpha^{s_\alpha s_\beta(\pi)} -                                   
s_\alpha w_{\alpha,\beta}(G_\alpha^{s_\alpha s_\beta(\pi)})
r_\beta^{\pi} r_\alpha^{\pi}\cdots r_\beta^{s_\alpha s_\beta(\pi)} \\
&=
r_\beta^{\pi} r_\alpha^{\pi}\cdots r_\beta^{s_\alpha s_\beta(\pi)} - 
r_\beta^{\pi} r_\alpha^{\pi}\cdots r_\beta^{s_\alpha s_\beta(\pi)} \\
&= 0. 
\end{align*}
Multiplying again on the right by $r_\alpha^{s_\beta(\pi)}$ and noting that $r_\alpha^{s_\alpha s_\beta(\pi)}r_\alpha^{s_\beta(\pi)} = G_\alpha^{s_\beta(\pi)} = 1$, we find the braid relation for $s_\beta(\pi)$:
\begin{align*}
r_\beta^{\pi} \cdots r_\beta^{s_\alpha s_\beta(\pi)} r_\alpha^{s_\beta(\pi)} - 
r_\alpha^{s_\alpha s_\beta(\pi)} \cdots r_\alpha^{\pi} r_\beta^{s_\beta(\pi)}.
\end{align*}
This concludes the claim for exceptional weights.

We move on, and assume $\lambda$ is parabolic with respect to $W^{\alpha,\beta}$.
If the stabilizer of $\lambda$ has 1 element, or is
$W^{\alpha, \beta}$ itself, then $\lambda$ was a standard parabolic weight 
with respect to $\{\alpha,\beta\}$ and we are done, as the braid relation
shows that the difference in \eqref{diff} is zero.

Thus, assume that $\lambda$ is a parabolic weight, but not a standard parabolic weight. 
Then there is a unique
$1\le t < m$ so that $\lambda_{t+1} = s_{\alpha_t}\cdots s_{\alpha_1}(\lambda)$ is a 
standard parabolic
weight with $s_{\alpha_{t+1}}(\lambda_{t+1})=\lambda_{t+1} $. We will swap
$\alpha, \beta$ if it happens that $t \ge \frac{m}{2}$, which has the effect of
changing $t$ to $m-t-1$. From now on, we assume $t < \frac{m}{2}$.

Define $\lambda_{-i} = \lambda_{i+1}, \mu_{-i} = \mu_{i+1}$, 
and multiply the difference in \eqref{diff} on the right by
$r_{\alpha_1}^{\lambda_{-1}}\cdots r_{\alpha_t}^{\lambda_{-t}}$. The two terms that appear
 are grouped as follows:
\begin{align*}
&r_{\alpha_m}^{\lambda_m}\cdots r_{\alpha_{t+1}}^{\lambda_{t+1}}
(r_{\alpha_t}^{\lambda_t}\cdots r_{\alpha_1}^{\lambda_1}
r_{\alpha_1}^{\lambda_{-1}}\cdots r_{\alpha_t}^{\lambda_{-t}})- \\
&r_{\beta_m}^{\mu_m}\cdots r_{\beta_{m-t+1}}^{\mu_{m-t+1}}
(r_{\beta_{m-t+1}}^{\mu_{m-t+1}}\cdots r_{\beta_1}^{\mu_1}
r_{\alpha_1}^{\lambda_{-1}}\cdots r_{\alpha_t}^{\lambda_{-t}}).
\end{align*}
As $\beta_1\not = \alpha_1$, the last $m$ entries of the second term
alternate between $\alpha$ and $\beta$, and start at the parabolic weight 
$\lambda_{-t} = \lambda_{t+1}$. Thus, they may be switched using the braid relation 
to the following:
\begin{align*}
&r_{\alpha_m}^{\lambda_m}\cdots r_{\alpha_{t+1}}^{\lambda_{t+1}}
(r_{\alpha_t}^{\lambda_t}\cdots r_{\alpha_1}^{\lambda_1}
r_{\alpha_1}^{\lambda_{-1}}\cdots r_{\alpha_t}^{\lambda_{-t}})- \\
&r_{\beta_m}^{\mu_m}\cdots r_{\beta_{m-t+1}}^{\mu_{m-t+1}}
(r_{\beta_{m-t+1}}^{\mu_{-(m-t+1)}}\cdots r_{\beta_m}^{\mu_{-m}}
r_{\alpha_m}^{\lambda_{m}}\cdots r_{\alpha_{t+1}}^{\lambda_{t+1}}).
\end{align*}
We combine the last $2t$ entries in the first term and the first $2t$ entries
in the second term to simplify this expression,
\begin{align*}
&r_{\alpha_m}^{\lambda_m}\cdots r_{\alpha_{t+1}}^{\lambda_{t+1}} \cdot P - \\
&P' \cdot r_{\alpha_m}^{\lambda_m}\cdots r_{\alpha_{t+1}}^{\lambda_{t+1}}
\end{align*}
where,
\begin{align*}
P=& \prod_{i=1}^t s_{\alpha_t}\cdots s_{\alpha_{i+1}} 
(G_{\alpha_i}^{\lambda_{-i}}), \\
P'=& \prod_{j=m-t+1}^{m} s_{\beta_m}\cdots s_{\beta_{j+1}}
(G_{\beta_j}^{\mu_{-j}}).
\end{align*}

Using the commutativity relation between $r_{\alpha}^\lambda$ and elements of $\lqA$
we find that the above expression is equal to,
\begin{align*}
&s_{\alpha_m}\cdots s_{\alpha_{t+1}}(P) \cdot 
r_{\alpha_m}^{\lambda_m}\cdots r_{\alpha_{t+1}}^{\lambda_{t+1}} + \\
&s_{\alpha_m}\cdots s_{\alpha_{t+2}}(h\demazure_{\alpha_{t+1}}(P)) \cdot 
r_{\alpha_m}^{\lambda_m}\cdots r_{\alpha_{t+2}}^{\lambda_{t+2}} - \\
&P'\cdot r_{\alpha_m}^{\lambda_m}\cdots r_{\alpha_{t+1}}^{\lambda_{t+1}} \cdot
\end{align*}

We claim that 
$s_{\alpha_m}\cdots s_{\alpha_{t+1}}(P) = P'$.
We can use the permutation $p$ from
lemma \ref{permute} to show that the $i$-th term in the product expression
for $s_{\alpha_{m}}\cdots s_{\alpha_{t+1}}(P)$ is the same as the $j$-th term in the
expression for $P'$, where $j=p(i)$. In fact, let $j=m-i+1 = p(i)$. 
Then the corresponding terms are exactly,
\begin{align*}
s_{\alpha_m}\cdots s_{\alpha_{i+1}} G_{\alpha_i}^{\lambda_{-i}}, \\
s_{\beta_m}\cdots s_{\beta_{j+1}}G_{\beta_j}^{\mu_{-j}}.
\end{align*}
The \emph{braid relation} for $G$ axiomatizes the above equality.

Now, $\mu_{-j} = w_{\alpha,\beta}(\lambda_{-i})$ and $\beta_j = w_{\alpha,\beta}(\alpha_i)$.
All in all, the difference in \eqref{diff} simplifies to,
\begin{align*}
s_{\alpha_m}\cdots s_{\alpha_{t+2}}(h\demazure_{\alpha_{t+1}}(P)) \cdot 
r_{\alpha_m}^{\lambda_m}\cdots r_{\alpha_{t+2}}^{\lambda_{t+2}}.
\end{align*}

The above expression has $m-t-1$ terms of the form $r_\alpha^\lambda$, and we can replace $t$ of them
(using the assumption that $t < \frac{m}{2}$)
after we multiply on the right by $r_{\alpha_t}^{\lambda_{-t}}\cdots r_{\alpha_1}^{\lambda_{-1}}$.

To summarize, let's define the following non-zero element of $\qA_\lambda$:
\begin{align*}
R = r_{\alpha_1}^{\lambda_{-1}}\cdots r_{\alpha_t}^{\lambda_{-t}}
r_{\alpha_t}^{\lambda_{t-1}}\cdots r_{\alpha_1}^{\lambda_1}.
\end{align*}
We have shown that the following relation holds in $\qhecke$:
\begin{align*}
&\left(r_{\alpha_m}^{\lambda_m} \cdots r_{\alpha_1}^{\lambda_1} - 
r_{\beta_m}^{\mu_m} \cdots r_{\beta_1}^{\mu_1}\right) R
\\ =& 
s_{\alpha_m}\cdots s_{\alpha_{t+2}}(h\demazure_{\alpha_{t+1}}(P)) \cdot 
r_{\alpha_m}^{\lambda_m}\cdots r_{\alpha_{2t+2}}^{\lambda_{2t+2}} \cdot R.
\end{align*}
The relation that $\qhecke$ has no right $\lqA$-torsion implies:
\begin{align*}
&r_{\alpha_m}^{\lambda_m} \cdots r_{\alpha_1}^{\lambda_1} - 
r_{\beta_m}^{\mu_m} \cdots r_{\beta_1}^{\mu_1} 
\\ =& 
s_{\alpha_m}\cdots s_{\alpha_{t+2}}(h\demazure_{\alpha_{t+1}}(P)) \cdot 
r_{\alpha_m}^{\lambda_m}\cdots r_{\alpha_{2t+2}}^{\lambda_{2t+2}},
\end{align*}
where again,
\begin{align*}
P=& \prod_{i=1}^t s_{\alpha_t}\cdots s_{\alpha_{i+1}} 
(G_{\alpha_i}^{\lambda_{-i}}).
\end{align*}

\vspace{10mm}

Finally, we show the assertion for reduced expressions. Let $w\in W$ with 
$\ell(w)=n$, and take two expressions 
$w = s_{\alpha_n}\cdots s_{\alpha_1} = s_{\beta_n}\cdots s_{\beta_1}$ of minimal length.
We show:
\begin{align*}
r_{\alpha_n}^{\lambda_n} \cdots r_{\alpha_1}^{\lambda_1} - 
r_{\beta_n}^{\mu_n} \cdots r_{\beta_1}^{\mu_1}\in \mathscr{F}^{n-1},
\end{align*}
by reducing to a smaller length case, or by using a braid relation. We need to 
apply the following lemma, which follows directly from lemma \ref{permute}, possibly many times.
\begin{lemma}
Let $u\in W$ with $\ell(u)=m$ and consider two reduced expressions 
$u = s_{\delta_m}\cdots s_{\delta_1} = s_{\gamma_m}\cdots s_{\gamma_1}$ in $W$.
Then there is a unique $1\le i_0 \le m$ with
\begin{align*}
\delta_1 =& s_{\gamma_1}\cdots s_{\gamma_{i_0-1}}(\gamma_{i_0}), \\
s_{\gamma_{i_0-1}}\cdots s_{\gamma_1} s_{\delta_1} =&
s_{\gamma_{i_0}} \cdots s_{\gamma_1}.
\end{align*}
\end{lemma}

Applying the lemma directly to the two expressions we have for $w$, we see if
$i_0< n$, then by induction we have
\begin{align*}
r_{\beta_{i_0}}\cdots r_{\beta_1} - r_{\beta_{i_0-1}}\cdots r_{\beta_1}r_{\alpha_1}
\in \mathscr{F}^{i_0-1}.
\end{align*}
Though we drop the weights $\mu_i, \lambda_i$ the reader may check this does no harm.

By the inductive hypothesis, 
\begin{align*}
r_{\alpha_n}\cdots r_{\alpha_2} - 
r_{\beta_n}\cdots r_{\beta_{i_0+1}}r_{\beta_{i_0-1}}\cdots r_{\beta_1}\in \mathscr{F}^{n-2}.
\end{align*}
Thus,
\begin{align*} 
&
\left(
r_{\alpha_n}\cdots r_{\alpha_1} - 
r_{\beta_n}\cdots r_{\beta_{i_0+1}}r_{\beta_{i_0-1}}\cdots r_{\beta_1} r_{\alpha_1}
\right) + \\
& 
\left(
r_{\beta_n}\cdots r_{\beta_{i_0+1}}r_{\beta_{i_0-1}}\cdots r_{\beta_1} r_{\alpha_1} -
r_{\beta_{n}}\cdots r_{\beta_1} 
\right)
\in \mathscr{F}^{n-1}.
\end{align*}

We assume $i_0=n$, or
\begin{align*}
\alpha_1 &= s_{\beta_1}\cdots s_{\beta_{n-1}}(\beta_n), \\
s_{\beta_{n-1}}\cdots s_{\beta_1}s_{\alpha_1} &=
s_{\beta_n}\cdots s_{\beta_1} = w.
\end{align*}
Similar to the argument above, we have by induction,
\begin{align*}
r_{\beta_{n-1}}\cdots r_{\beta_1} - r_{\alpha_n}\cdots r_{\alpha_2}\in 
\mathscr{F}^{n-2}.
\end{align*}
Thus, the following two assertions are equivalent,
\begin{align*}
&r_{\beta_n}\cdots r_{\beta_1} - r_{\alpha_n}\cdots r_{\alpha_1}\in \mathscr{F}^{n-1}, \\
&r_{\beta_n}\cdots r_{\beta_1} - r_{\beta_{n-1}}\cdots r_{\beta_1}r_{\alpha_1}\in 
\mathscr{F}^{n-1}.
\end{align*}

We now apply the lemma above to the second expression, finding an $i_0$ with,
\begin{align*}
\beta_1 = s_{\alpha_1}s_{\beta_1}\cdots s_{\beta_{i_0-2}}(\beta_{i_0-1}).
\end{align*}
Again, either $i_0 <n$ in which case we apply the induction to show the claim,
or we show the following two assertions are equivalent,
\begin{align*}
&r_{\beta_n}\cdots r_{\beta_1} - r_{\beta_{n-1}}\cdots r_{\beta_1}r_{\alpha_1}\in 
\mathscr{F}^{n-1} \\
&r_{\beta_{n-1}}\cdots r_{\beta_1}r_{\alpha_1} - 
r_{\beta_{n-2}}\cdots r_{\beta_1}r_{\alpha_1} r_{\beta_1}\in 
\mathscr{F}^{n-1}.
\end{align*}
Using the same trick we show either the second claim or the equivalence of the following
two assertions,
\begin{align*}
&r_{\beta_{n-1}}\cdots r_{\beta_1}r_{\alpha_1} - 
r_{\beta_{n-2}}\cdots r_{\beta_1}r_{\alpha_1} r_{\beta_1}\in 
\mathscr{F}^{n-1}, \\
&r_{\beta_{n-3}}\cdots r_{\beta_1}r_{\alpha_1} r_{\beta_1}r_{\alpha_1}- 
r_{\beta_{n-2}}\cdots r_{\beta_1}r_{\alpha_1} r_{\beta_1}\in 
\mathscr{F}^{n-1}.
\end{align*}
At this point, if $m_{\alpha,\beta} = 3$ we are done due to the above
proof for the braid relation. If $m_{\alpha,\beta}$ is larger, we keep applying this algorithm
to eventually find a braid relation.

This finishes the proof of the lemma.
\end{proof}

\begin{corollary} \label{structure}
Let $B$ be a set of reduced expressions $r_{\alpha_n}\cdots r^{\lambda}_{\alpha_n}$ 
so that every $w\in W$ is represented exactly once. Then $B$ generates
$\qhecke1_\lambda$ as a right $\qA_\lambda$-module.
\end{corollary}

Let $gr\qhecke$ be the graded algebra associated to the filtration $(\mathscr{F}^n)$.
We wish to describe the structure of $gr \qhecke$.

Let ${}^0\mathscr{H}^f$ be the finite nil-Hecke algebra. This is the algebra 
with generators $r_{\alpha}, \alpha\in \Pi$, satisfying:
\begin{align*}
r_{\alpha}^2 &= 0, \\
\cdots r_{\beta} r_{\alpha} &= \cdots r_{{\alpha}}r_{\beta},
\text{with $m_{\alpha, \beta}$ terms.}
\end{align*}

We form the wreath product algebra
\begin{align*}
\lqA \wr {}^0\mathscr{H}^f,
\end{align*}
which as a $k$ vector space is given by the tensor product, 
$\lqA\otimes_k {}^0\mathscr{H}^f$. We give the multiplication by setting,
\begin{align*}
1\otimes r_{s_\alpha} \cdot f\otimes 1 = s_\alpha(f)\otimes r_{s_{\alpha}}.
\end{align*}

There is a natural surjective morphism
\begin{align*}
\lqA\wr {}^0\mathscr{H}^f\to gr\qhecke.
\end{align*}
We say that $\qhecke$ has the PBW property if this morphism is an isomorphism.

\begin{theorem}\label{PBW}
The following assertions hold:
\begin{itemize}
\item
$\qhecke$ satisfies the PBW property.
\item
For every $\lambda\in \qT $,
$\qhecke1_\lambda$ is a free right $\qA _\lambda $-module with basis $B$.
\end{itemize}
\end{theorem}
\begin{proof}
The first two assertions are equivalent thanks to the generating family $B$ mentioned 
in the above corollary.

\begin{lemma}\label{splitting}
Given a family $G$ satisfying the conditions of section \ref{conditions}, there exists
a splitting family $F = (F_\alpha^\lambda)$, $F_\alpha\in \iA_\lambda$, which satisfy the following conditions:
\begin{enumerate}
\item
One of $F_\alpha^\lambda$ or $F_\alpha^{s_\alpha(\lambda)}$ is equal to 1.
\item
$F_\alpha^\lambda\cdot s_\alpha(F_\alpha^{s_\alpha(\lambda)}) = G_\alpha^\lambda$.
\end{enumerate}
\end{lemma}

\begin{remark}
This lemma takes the place of the splitting $Q_{i,j}(u,u') = P_{i,j}(u,u')P_{j,i}(u',u)$ in \cite[Section 3.2.3]{rouquier-qha}.
\end{remark}
\begin{proof}
Fix $\lambda\in \iT,\alpha,\beta\in \Pi$ distinct. 
If $s_\alpha(\lambda) = \lambda$, we put $F_\alpha^\lambda = F_\alpha^{s_\alpha(\lambda)}=1$.
Note, in this case, $G_\alpha^\lambda=1$.
We see that $w_{\alpha,\beta}s_\alpha(\lambda) = w_{\alpha,\beta}(\lambda)$, and because $s_{w_{\alpha,\beta}s_\alpha(\alpha)}w_{\alpha,\beta} = w_{\alpha,\beta}s_\alpha$, we have  $F_{w_{\alpha,\beta}s_\alpha(\alpha)}^{w_{\alpha,\beta}s_\alpha(\lambda)} = 
F_{w_{\alpha,\beta}s_\alpha(\alpha)}^{w_{\alpha,\beta}(\lambda)}$, so this choice is consistent with the braid relation.

Assume $s_\alpha(\lambda)\not=\lambda$, and set $F_\alpha^\lambda = G_\alpha^\lambda, F_\alpha^{s_\alpha(\lambda)}=1$.
Consider the set,
\begin{align}\label{four-weights}
\{\lambda, s_\alpha(\lambda), w_{\alpha,\beta}s_\alpha(\lambda), w_{\alpha,\beta}(\lambda)\}.
\end{align}
As $s_\alpha(\lambda)\not=\lambda$, we have $w_{\alpha,\beta}s_\alpha(\lambda)\not= w_{\alpha,\beta}(\lambda)$.
In accordance with the braid relations, we set
\begin{align}\label{Fdefine}
F_{w_{\alpha,\beta}s_\alpha(\alpha)}^{w_{\alpha,\beta}s_\alpha(\lambda)} 
:=&
w_{\alpha,\beta}s_\alpha(F_\alpha^\lambda) = w_{\alpha,\beta}s_\alpha(G_\alpha^\lambda),
\\ \nonumber
F_{w_{\alpha,\beta}s_\alpha(\alpha)}^{w_{\alpha,\beta}(\lambda)}
:=&
w_{\alpha,\beta}s_\alpha(F_\alpha^{s_\alpha(\lambda)}) = 1.
\end{align}
If $m_{\alpha,\beta}$ is odd, then $w_{\alpha,\beta}s_\alpha(\beta) = \beta$  and the four pairs
\begin{align}\label{four-pairs}
\{
(\lambda,\alpha), (s_\alpha(\lambda),\alpha), 
(w_{\alpha,\beta}(\lambda),w_{\alpha,\beta}s_\alpha(\alpha)),
(w_{\alpha,\beta}s_\alpha(\lambda),w_{\alpha,\beta}s_\alpha(\alpha))
\},
\end{align}
are distinct, so we have not defined any element of $F$ twice.
If $m_{\alpha,\beta}$ is even and $\lambda = w_{\alpha,\beta}s_\alpha(\lambda)$, then the two sides of \eqref{Fdefine} are already equal by the braid relation for $G_\alpha^\lambda = F_\alpha^\lambda$.
Thus, we have defined the two elements, $F_\alpha^\lambda, F_\alpha^{s_\alpha(\lambda)}$ twice, but with the same values each time.
If $m_{\alpha,\beta}$ is even and $\lambda = w_{\alpha,\beta}(\lambda)$, then because $w_{\alpha,\beta}$ has even length, it is not a reflection, thus $\lambda$ is not a parabolic weight.
This means that $\lambda$ is $\alpha$-exceptional so all four values in \eqref{Fdefine} are 1.

Now, let $\gamma\in \Pi$ be distinct from $\alpha,\beta$ and define 
$F_{w_{\alpha,\gamma}s_\alpha(\alpha)}^{w_{\alpha,\gamma}s_\alpha(\lambda)},
F_{w_{\alpha,\gamma}s_\alpha(\alpha)}^{w_{\alpha,\gamma}(\lambda)}$ as above.
If either of $m_{\alpha,\beta}$ or $m_{\alpha,\gamma}$ are odd, then the values of the $F$-terms are well defined.
Assuming that $m_{\alpha,\beta}, m_{\alpha,\gamma}$ are both even we see that if $w_{\alpha,\gamma}(\lambda) = w_{\alpha,\beta}(\lambda$ then as $F_\alpha^{s_\alpha(\lambda)} =1$, we indeed have $w_{\alpha,\beta}s_\alpha(F_\alpha^{s_\alpha(\lambda)}) = w_{\alpha,\gamma}s_\alpha(F_\alpha^{s_\alpha(\lambda)})$.

In case $w_{\alpha,\beta}s_\alpha(\lambda) = w_{\alpha,\gamma}s_\alpha(\lambda)$ we have the braid relation:
\begin{align*}
w_{\alpha,\beta}s_\alpha(G_\alpha^\lambda) = 
G_\alpha^{w_{\alpha,\beta}s_\alpha(\lambda)} =
G_\alpha^{w_{\alpha,\gamma}s_\alpha(\lambda)} =
w_{\alpha,\gamma}s_\alpha(G_\alpha^\lambda).
\end{align*}
Thus, to show that no contradiction forms from these choices it is enough to consider the case
\begin{align}
w_{\alpha,\gamma}(\lambda) =& w_{\alpha,\beta}s_\alpha(\lambda),\\
w_{\alpha,\gamma}s_\alpha(\alpha) =& w_{\alpha,\beta}s_\alpha(\alpha) = \alpha.
\end{align}
In the case $m_{\alpha,\beta} = m_{\alpha,\gamma} = 2$ we find that if $s_\alpha s_\gamma(\lambda) = s_\beta(\lambda)$ then $s_{\alpha}s_\beta s_\gamma(\lambda) = \lambda$.
As $s_\gamma(\omega_\alpha) = s_\beta(\omega_\alpha) = \omega_\alpha$ we find,
\begin{align*}
\omega_\alpha(\lambda) =& \omega_\alpha(s_\alpha s_\beta s_\gamma(\lambda)) \\
=& \omega_\alpha(\lambda)^{-1}.
\end{align*}
It follows that $\omega_\lambda(\lambda) = \pm 1$.
As $m_{\alpha,\beta} = m_{\alpha,\gamma}=2$ this shows that $\lambda$ is in fact $s_\alpha$-invariant, i.e. this case never happens.

We are left to consider the rank 3 root systems with the following cases
$m_{\alpha,\beta}=2, m_{\alpha,\gamma}=4,6$ and $m_{\alpha,\beta}=4,6$, $m_{\alpha,\gamma}=2$.
We must show that $F_\alpha^{w_{\alpha,\gamma}(\lambda)} = F_\alpha^{w_{\alpha,\beta}s_\alpha(\lambda)}$, which by definition means we must show that,
\begin{align*}
w_{\alpha,\gamma}s_\alpha(F_\alpha^{s_\alpha(\lambda)}) = w_{\alpha,\beta}s_\alpha(F_\alpha^{\lambda}).
\end{align*}
As we have already defined $F_\alpha^{s_\alpha(\lambda)} = 1$, we must show that $F_\alpha^{\lambda} = G_\alpha^\lambda = 1$. 

When $m_{\beta,\gamma} = 2$, we see $\omega_\gamma$ is both $s_\alpha$ and $s_\beta$ invariant.
Thus $s_\gamma s_\alpha(\lambda) = w_{\alpha,\beta} s_\alpha(\lambda)$ implies that $\omega_\gamma(\lambda)^{-1} = \omega_\gamma(\lambda)$, and hence $s_\gamma(\lambda) = \lambda$.
From this we deduce that $w_{\alpha,\beta}(\lambda) = \lambda$. 
As the length of $w_{\alpha,\beta}$ is even and $s_\alpha(\lambda)\not=\lambda$ we deduce that $\lambda$ is $\alpha$-exceptional and so $G_\alpha^\lambda=1$ as we desired.

The other cases arise from the simply connected root datum associated with $B_3$ and $C_3$.
Consider the simply connected root datum associated to $B_3$. 
Let $\Pi=\{\alpha,\beta,\gamma\}$, where $\alpha$ is the short root, $m_{\alpha,\beta}=4$, and $m_{\alpha,\gamma}=2$. 
By an explicit calculation with the element $\lambda = (x,y,z)\in (k^*)^3$ corresponding to $(\check{\alpha}\otimes x)\cdot (\check{\beta}\otimes y)\cdot (\check{\gamma}\otimes z)\in \iT\otimes k^*$, we find that the only elements satisfying $s_\gamma s_\alpha(\lambda) = s_\beta s_\alpha s_\beta(\lambda)$ and $s_\alpha(\lambda)\not=\lambda$ are of the form $(i,1,-1)$ where $i$ is a square root of $(-1)$. 
In that case the Cayley graph of the action of $s_\alpha, s_\beta, s_\gamma$ looks like,
\begin{align*}
{}^{s_\beta , s_\gamma} \circlearrowleft
(i,1,-1) \xleftrightarrow{s_\alpha} (-i,1,-1) \circlearrowright
{}^{s_\beta, s_\gamma}.
\end{align*}
As $m_{\alpha,\beta} = 4$ it is clear that $\lambda$ is $\alpha$-exceptional and so $G_\alpha^\lambda = 1$ as desired.

A similar calculation for the simply connected root datum associated to $C_3$ shows that every weight $\lambda$ with $w_{\alpha,\gamma}(\lambda) = w_{\alpha,\beta}s_\alpha(\lambda)$ are in fact $s_\alpha$-invariant.

This shows that we may define $F_\alpha^\lambda$ consistently.
\end{proof}

Now take a splitting family $F_\alpha^\lambda\in \qA _\lambda$ 
for $G$.

For 
$s_\alpha(\lambda)=\lambda$ we let $r_\alpha^\lambda$ act as
$h\demazure_\alpha 1_\lambda$. 
Otherwise we let $r^\lambda_\alpha$ act as $ s_\alpha F_\alpha^\lambda 1_\lambda$.
To show this representation is well defined we only need to check the relations.

The only difficult relation is the braid relation in the case where $\lambda$
is a parabolic weight with respect to $W^{\alpha,\beta}$, but is fixed 
by only one of the weights. 

For this case, suppose $\lambda$ is $s_\alpha$ invariant and not $s_\beta$ invariant.
We set 
\begin{align*}
\alpha_i =& \begin{cases}  \alpha, \text{ if $i$ odd}, \\ \beta, \text{ if $i$ even}\end{cases} \\
\beta_i =& \begin{cases}  \alpha, \text{ if $i$ even}, \\ \beta, \text{ if $i$ odd}\end{cases}.
\end{align*}
The relevant relation we must show is equivalent to
\begin{align*}
s_{\alpha_m} s_{\alpha_{m-1}}\cdots s_{\alpha_2}\demazure_{\alpha_1} = 
\demazure_{\beta_m} s_{\beta_{m-1}}\cdots s_{\beta_1}.
\end{align*}

If we consider $\demazure_\alpha$ as given by the fraction, $\dfrac{1-s_{\alpha}}{1-e^{-\alpha}}$,
then the relation
\begin{align*}
\demazure_{\alpha_m} = (w_\ell s_\alpha) \demazure_{\alpha_1} (w_\ell s_\alpha)^{-1}
\end{align*}
makes the desired relation above obvious.

The above morphism defines a faithful representation of $\qhecke$
on $\lqA$.

The image of the set $B\subset \qhecke1_\lambda$ gets mapped to
$(\lqA \wr W)1_\lambda$, and is linearly independent over 
$1_\lambda\lqA 1_\lambda$, just as in \cite[Proposition 3.8]{rouquier-qha}, with $k[X_1,\cdots X_n]$ replaced by $\iA$.
\end{proof}

\subsection{Isomorphism class of $\qhecke$} 
The main result of this section shows that the isomorphism class of $\qhecke$ is invariant under multiplying the data $G$ by invertible functions which also satisfy braid and reflexive relations.

\begin{theorem} \label{isomorphism-class}
Let $G = (G_\alpha^\lambda)$ and $H=(H_\alpha^\lambda)$ be datum satisfying the conditions from section \ref{conditions}.
Suppose $(g_\alpha^\lambda)_{\lambda\in \iT, \alpha\in\Pi}$ is the set of functions
$g_\alpha^\lambda = H_\alpha^\lambda / G_\alpha^\lambda$, and suppose that the $g_\alpha^\lambda$ are invertible rational functions, $g_\alpha^\lambda\in (\iA_\lambda)^*$.
Then there is an isomorphism,
\begin{align*}
\qheckeh\to \qhecke
\end{align*}
\end{theorem}

\begin{proof}

Suppose $G$ and $H$ are sets of datum satisfying the conditions from section \ref{conditions}.
Suppose, further, that $g_\alpha^\lambda = H_\alpha^\lambda / G_\alpha^\lambda$ is a unit in $\iA_\lambda$. 
By the splitting lemma, \ref{splitting}, there exists a splitting family $(F) = (F_\alpha^\lambda)$ for $(g_\alpha^\lambda)$. 
As each $F_\alpha^\lambda$ is either $1$ or $g_\alpha^\lambda$, we find that $F^\lambda_\alpha$ is also invertible in $\iA_\lambda$.
Consider the elements,
\begin{align*}
\tau_\alpha^\lambda = r_\alpha^\lambda F_\alpha^\lambda\in \qhecke.
\end{align*}
From the proof of Theorem \ref{PBW}, we find the elements $\tau_\alpha^\lambda$ satisfy the same relations as $r_\alpha^\lambda\in \qheckeh$. Moreover they generate, along with $\lA$ the algebra $\qhecke$. This proves our claim.
\end{proof}

\section{Applications} \label{applications}

\subsection{Affine Hecke algebras and localized quiver Hecke algebras} \label{isomorphism}
Given a root datum $(X,Y,R,\check{R},\Pi)$ and set of parameters $c_\alpha\in \C^*, \alpha\in \Pi$, 
with a fixed $h_0\in \C$, we construct datum $G$ satisfying
the properties above, and an isomorphism
$\qhecke \to \lhecke$. 

If $s_\alpha(\lambda)=\lambda$ or $\lambda$ is $\alpha$-exceptional, let $G_\alpha^\lambda=1$. 
Otherwise, for $\lambda$ with 
$s_\alpha(\lambda)\not=\lambda$ let
\begin{align*}
G_\alpha^\lambda= 
(c_\alpha+q_\alpha P_{-\alpha})(P_{-\alpha}-c_\alpha)(-P_{-\alpha})^{-2}
\end{align*}

\begin{theorem}\label{proof}
The data $G$ constructed above satisfies the conditions of section \ref{conditions},
so $\qhecke$ is well defined.
Consider the map $\qhecke \to \lhecke$ which is the 
identity on $\lqA$, and on generators is given by:
\begin{align*}
r^\lambda_{\alpha}\mapsto  \begin{cases}
(c_\alpha+q_\alpha P_{-\alpha})^{-1} (T_{s_\alpha}-q_{s_\alpha})1_{\lambda},& 
\text{if $s_\alpha(\lambda)=\lambda$},\\
\left( \dfrac{P_{-\alpha}}{c_\alpha+P_{-\alpha}+hc_\alpha P_{-\alpha}}\right)
1_{s_\alpha(\lambda)}T_{s_\alpha} 1_\lambda , & \text{ if $\lambda$ is $\alpha$-exceptional.},\\
1_{s_\alpha(\lambda)}T_{s_\alpha} 1_\lambda , & \text{else.}
\end{cases}
\end{align*}
This map is well defined and it is an isomorphism.
\end{theorem}

\begin{proof}
We easily see that $G^\lambda_\alpha$ satisfies the associative property, 
and the braid relation follows from the Weyl group lemmas. 
It follows that $\qhecke$
is well defined. To check that the above map is well defined 
we must check the 4 relations from section \ref{relations} on the generators, and confirm that there is no
right $\qA_\lambda$-torsion in $\lhecke1_\lambda$.

Abusing notation, we use $r_\alpha^\lambda$ for its image in $\lhecke$.
From the definition of $\lhecke$ we have that 
$1_{s_\alpha(\nu)} r_\alpha^\lambda  = r_\alpha^\lambda 1_\nu =
\delta_{\lambda,\nu}r_\alpha^\lambda$.

We now check the quadratic relation,
\begin{align*}
r_{\alpha}^{s_{\alpha}(\lambda)}r_{\alpha}^{\lambda} = 
\begin{cases}
G_\alpha^\lambda
&\text{ if $s_\alpha(\lambda) \not = \lambda$,} \\
hr_\alpha^\lambda
&\text{ if $s_\alpha(\lambda)=\lambda$}.
\end{cases}
\end{align*}
First, suppose $s_\alpha(\lambda)\not =\lambda$. We have the quadratic relation
$T_{s_\alpha}^2 = (q_{s_\alpha}-1)T_{s_\alpha}+q_{s_\alpha}$. We multiply on the left and right by $1_\lambda$
to obtain,
\begin{align*}
1_\lambda T_{s_\alpha}^2 1_\lambda =& (q_{s_\alpha}-1)1_\lambda T_{s_\alpha} 1_\lambda +q_{s_\alpha} 1_\lambda,\\
=& (q_{s_\alpha}-1)c_\alpha(-P_{-\alpha})^{-1}1_\lambda + q_\alpha.
\end{align*}
On the other hand, we have:
\begin{align*}
1_\lambda T_{s_\alpha}^21_\lambda &= 
1_\lambda T_{s_\alpha} (1_\lambda+1_{s_\alpha(\lambda)})T_{s_\alpha} 1_\lambda, \\
=& 1_\lambda T_{s_\alpha}1_{s_\alpha(\lambda)}T_{s_\alpha} 1_\lambda + 
c_\alpha^2 (-P_{-\alpha})^{-2}.
\end{align*}
Equating the two expressions yields the equality:
\begin{align*}
1_\lambda T_{s_\alpha} 1_{s_\alpha(\lambda)}T_{s_\alpha} 1_\lambda = 
(c_\alpha+q_\alpha P_{-\alpha})(P_{-\alpha}-c_\alpha)(-P_{-\alpha})^{-2}.
\end{align*}
Consequently, we find that $1_{s_\alpha(\lambda)}T_{s_\alpha} 1_\lambda$ is invertible when $\lambda(P_{-\alpha})\not=c_\alpha, -q^{-1}c_\alpha$. 

Next, we verify the commutativity relation:
\begin{align*}
r_{\alpha}^{\lambda} f - s_{\alpha}(f)r_{\alpha}^{\lambda}
=\begin{cases} 0 &\text{, if $s_\alpha(\lambda)\not =\lambda$,}\\
h\demazure_\alpha(f) &\text{, if $s_\alpha(\lambda)=\lambda$}.
\end{cases}
\end{align*}
First, suppose $s_\alpha(\lambda)\not=\lambda$. We simply multiply the original
commutativity relation,
\begin{align*}
T_{s_\alpha} f - s_\alpha(f) T_{s_\alpha} = c_\alpha h\demazure(f),
\end{align*}
on the left by $1_{s_\alpha(\lambda)}$ and on the right
by $1_\lambda$. Since $1_\lambda 1_{s_\alpha(\lambda)}=0$, the claim follows.

Now suppose $s_\alpha(\lambda)=\lambda$. 
We check directly:
\begin{align*}
(T_\alpha -q_\alpha)f - s_\alpha(f)(T_\alpha - q_\alpha) 
=&
c_\alpha h\demazure_\alpha(f) - q_\alpha (f-s_\alpha(f)),\\
=&
(c_\alpha+q_\alpha P_{-\alpha})h\demazure_\alpha(f).
\end{align*}

One could also expand the expression,
\begin{align*}
(c_\alpha+q_\alpha P_{-\alpha})^{-1}(T_\alpha-q_\alpha)
(c_\alpha+q_\alpha P_{-\alpha})^{-1}(T_\alpha-q_\alpha)
\end{align*}
and verify the quadratic relation, $(r_\alpha^\lambda)^2=h r_\alpha^\lambda$, but we will use the induced representation of $\ihecke$ on $\iA$ for this and the braid relations.

Finally we verify the braid relations. The only standard parabolic subgroups
of the Coxeter group $(W^{\alpha,\beta},\{s_\alpha,s_\beta\})$ 
are $W^{\alpha,\beta}, \langle e\rangle, \langle s_\alpha\rangle, \langle s_\beta\rangle$.

Suppose that the stabilizer of $\lambda$ is the trivial
group $\langle e\rangle$. The element $r_\alpha^\lambda$ is given by
$1_{s_\alpha(\lambda)}T_{s_\alpha}1_\lambda$. In this case, with 
$\lambda' = \cdots s_\alpha s_\beta s_\alpha(\lambda)$, we have
\begin{align*}
1_{\lambda'} \cdots T_{s_\alpha} T_{s_\beta} T_{s_\alpha} 1_\lambda = 
1_{\lambda'}\cdots T_{s_\alpha} 1_{s_\beta s_\alpha(\lambda)} 
T_{s_\beta} 1_{s_\alpha(\lambda)} T_{s_\alpha} 1_\lambda,
\end{align*}
and similarly for $\cdots T_{s_\beta} T_{s_\alpha} T_{s_\beta}$. Thus the braid relation for 
$T_{s_\alpha}, T_{s_\beta}$ yields the braid relation between $r_\alpha,r_\beta$.

Consider, now, the case where the stabilizer of $\lambda$ is $s_\alpha$. In this case, we also have
\begin{align*}
1_{\lambda'}\cdots T_{s_\alpha}T_{s_\beta}T_{s_\alpha} 1_\lambda = 
1_{\lambda'}\cdots T_{s_\alpha}1_{s_\beta s_\alpha(\lambda)}
T_{s_\beta}1_{s_\alpha(\lambda)}T_{s_\alpha} 1_\lambda.
\end{align*}
Replacing the rightmost $T_{s_\alpha}$ with 
$(c_\alpha + q_{s_\alpha}P_{-\alpha})^{-1}(T_{s_\alpha}-q_{s_\alpha})$
and using the commutativity relation yields the desired result.

Finally, suppose that $\operatorname{stab}_{W^{\alpha,\beta}}(\lambda)=W^{\alpha,\beta}$.
We will use the Demazure-Lusztig representation of $\ihecke$ from section \ref{lusztig-rep}.

Recall equation \eqref{t-q}, which gives the formula for the action of $\widehat{T_\alpha-q_\alpha}$ on $\iA$:
\begin{align*}
\widehat{T_\alpha-q_\alpha}: f\mapsto 
(c_\alpha+q_\alpha P_{-\alpha})h\demazure_\alpha(f).
\end{align*}

We extend the action of $\ihecke$ on $\qA$
to an action of $\lhecke1_\lambda$ on $\qA_\lambda$,
and find $r_\alpha^\lambda = h\demazure_\alpha$. As the Demazure-Lusztig representation is faithful this shows the braid relation
between $r_\alpha^\lambda, r_\lambda^\beta$, as well as the quadratic relation $r_\alpha^\lambda r_\alpha^\lambda = h r_\alpha^\lambda$.

From the structure theory of $\lhecke$ we see it has no polynomial torsion,
and the same PBW basis, by the same Demazure-Lusztig representation, thus the map in question is an isomorphism.

\end{proof}
\subsection{Quiver Hecke algebras}\label{graded-hecke-algebra}
In this section we define quiver Hecke algebras attached to simply connected semisimple root data as a subalgebra of $\qhecke$ defined in the previous section.
Let $h_0 = 0$, then $P_\alpha + P_\beta = P_{\alpha+\beta}$.
Pick a choice of parameters $c_\alpha\in k^*$.
Recall that there are no exceptional weights in this case, as every weight is conjugate to a standard parabolic weight.

Recall the data $G$ associated to a simply connected semisimple root datum $(X,Y, R,\check{R},\Pi)$:
\begin{align*}
G_\alpha^\lambda = \begin{cases}
1 & \text{if $s_\alpha(\lambda) = \lambda$,} \\
(P_\alpha-c_\alpha) (P_\alpha+c_\alpha) (P_\alpha)^{-2} & \text{else.}
\end{cases}
\end{align*}

Suppose the characteristic of $k$ is not $2$.
Define the data $H$ and $g$ as follows:
\begin{align*}
H_\alpha^\lambda =& \begin{cases}
1 & \text{if $s_\alpha(\lambda) = \lambda$,} \\
c_\alpha - P_\alpha & \text{if $\langle \lambda, \alpha\rangle = c_\alpha$,} \\
c_\alpha + P_\alpha & \text{if $\langle \lambda, \alpha \rangle = -c_\alpha$,} \\
1 & \text{else,} 
\end{cases} \\
g_\alpha^\lambda =& \begin{cases}
1 & \text{if $s_\alpha(\lambda) = \lambda$}, \\
(P_\alpha+c_\alpha)^{-1}(P_\alpha)^2 & \text{if $\langle \lambda, \alpha\rangle =  c_\alpha$}, \\
(P_\alpha-c_\alpha)^{-1} (P_\alpha)^2 & \text{if $\langle \lambda, \alpha\rangle = -c_\alpha$}, \\
(P_\alpha-c_\alpha)^{-1}(P_\alpha+c_\alpha)^{-1}(P_\alpha)^2 & \text{else,}
\end{cases}
\end{align*}
In the case that the characteristic of $k$ is $2$, define $H$ and $g$ instead as,
\begin{align*}
H_\alpha^\lambda =& \begin{cases}
1 & \text{if $s_\alpha(\lambda) = \lambda$,} \\
(c_\alpha - P_\alpha)^2 & \text{if $\langle \lambda, \alpha\rangle = c_\alpha$,} \\
1_\lambda & \text{else,} 
\end{cases} \\
g_\alpha^\lambda =& \begin{cases}
1_\lambda & \text{if $s_\alpha(\lambda) = \lambda$}, \\
(P_\alpha)^2 & \text{if $\langle \lambda, \alpha\rangle =  c_\alpha$}, \\
(P_\alpha-c_\alpha)^{-1}(P_\alpha+c_\alpha)^{-1}(P_\alpha)^2 & \text{else,}
\end{cases}
\end{align*}

\begin{proposition}
The data $H$ satisfies the assumptions of section \ref{conditions}.
Moreover, the algebras $\qhecke$ and $\qheckeh$ are isomorphic.
\end{proposition}
\begin{proof}
Simply apply theorem \ref{isomorphism-class} to the datum $G, H, g$.
\end{proof}
The advantage of the datum $H_\alpha^\lambda$ is that it is contained in the image of $k[\dT]\hookrightarrow k[\dT]_\lambda$.
This allows us to define the following subalgebra of the algebras $\qheckeh$.
\begin{definition}
Suppose $H = (H_\alpha^\lambda)_{\alpha\in \Pi, \lambda\in \dT}$ is a datum which satisfies the conditions of section \ref{conditions}, and for which $H_\alpha^\lambda$ is in the image of the inclusion $k[\dT]\hookrightarrow \aA_\lambda$.
Let $\grhecke$ be the subalgebra of $\qheckeh$ generated by the image of $k[\dT]\hookrightarrow \aA_\lambda$, for each $\lambda\in \dT$ and $r_\alpha^\lambda, \alpha\in \Pi, \lambda\in \dT$.
\end{definition}
\begin{proposition}
The inclusion $\grhecke\hookrightarrow \qhecke$ gives, via pullback, an equivalence from the category of finite representations of $\qhecke$ and the category of representations of $\grhecke$ for which $(P_\alpha - \langle \lambda, \alpha\rangle)^n1_\lambda$ acts by $0$ for large enough $n$.
\end{proposition}
\begin{proof}
It is plain that a $\dA$-module $V_\lambda$ with the property $(P_\alpha - \langle \lambda, \alpha\rangle)^n$ acts by zero for large enough $n$ has a unique extension to a module over ${\dA}_\lambda$ by letting elements $f^{-1}\in \dA_\lambda$ with $f\in \dA$ with $f(\lambda)\not=0$ act via the inverse of the action of $f$, which has only non-zero eigenvalues on $V_\lambda$.
The claim the follows, as any module over $\grhecke$ with the above property has a unique lift to a module over $\qheckeh$.
\end{proof}
We do one more change of variables to get the presentation of $\grhecke$ that we need to define a grading.
Let $\psi_\lambda: k[\dT]\to k[\dT]$ be the $W$-equivariant map given by,
\begin{align*}
\psi_\lambda(P_\alpha) = P_\alpha + \langle \lambda, \alpha\rangle .
\end{align*}
After this change of variables we use the notation $x_\alpha$ for the variable $P_\alpha$, thus for $f\in k[\dT]$, a polynomial in $\{P_\alpha\}_{\alpha\in \Pi}$, we consider $\psi_\lambda(f)$ a polynomial in $\{x_\alpha\}_{\alpha\in \Pi}$.

Let $\tilde{R}$ be the algebra with generating set $\{1_\lambda\}_{\lambda\in \dT}\cup \{x_\alpha^\lambda\}_{\alpha\in \Pi, \lambda\in \dT}\cup\{\tau_\alpha^\lambda\}_{\alpha\in \Pi, \lambda\in \dT}$ and relations:

\begin{gather*}
1_\lambda 1_{\lambda'} = 1_\lambda \delta_{\lambda,\lambda'}, \\
x_\alpha^\lambda 1_{\lambda'} = 1_{\lambda'} x_\alpha^\lambda = \delta_{\lambda,\lambda'}x_i^\lambda, \\
\tau_\alpha^\lambda 1_{\lambda'} = 1_{s_\alpha(\lambda')}\tau_\alpha^\lambda = \delta_{\lambda,\lambda'}\tau_\alpha^\lambda, \\
x_\alpha^\lambda x_\beta^\lambda = x_\beta^\lambda x_\alpha^\lambda, \\
\tau_\alpha^\lambda x_\beta^\lambda = x_{s_\alpha(\beta)}^{s_\alpha(\lambda)} \tau_\alpha^\lambda,
\text{ for $s_\alpha(\lambda)\not = \lambda$} , \\
\tau_\alpha^\lambda x_\beta^\lambda - x_{s_\alpha(\beta)}^{\lambda} \tau_\alpha^\lambda = 
c_\alpha \langle \beta, \check{\alpha}\rangle, \text{ for $s_\alpha(\lambda) = \lambda$}
\end{gather*}
along with the following quadratic relations,
\begin{align*}
\tau_\alpha^{s_\alpha(\lambda)} \tau_\alpha^\lambda = \begin{cases}
0 & \text{if  $s_\alpha(\lambda) = \lambda$}, \\
\psi_\lambda(H_\alpha^\lambda) & \text{else.}
\end{cases}
\end{align*}

Also give $\tilde{R}$ the following braid relation between $\tau_\alpha$ and $\tau_\beta$ for $\lambda$ a standard parabolic weight.
For $\alpha, \beta\in \Pi$ distinct and $\lambda\in \dT $ a \emph{standard parabolic weight}
with respect to $\{\alpha, \beta\}$, 
\begin{align*}
\tau_{\alpha_m}^{\lambda_m}\cdots \tau_{\alpha_1}^{\lambda_1} = 
\tau_{\beta_m}^{\mu_m}\cdots \tau_{\beta_1}^{\mu_1},
\end{align*}
where $m=m_{\alpha, \beta}$ is the order of $s_{\alpha}s_\beta$ in $W$, 
\begin{align*}
\alpha_i =& \begin{cases}
\alpha &\text{, $i$ odd,} \\
\beta & \text{, $i$ even,}
\end{cases} \\
\beta_i =& \begin{cases}
\beta &\text{, $i$ odd,} \\
\alpha & \text{, $i$ even,}
\end{cases} \\
\end{align*}
and $\lambda_i = s_{\alpha_{i-1}}\cdots s_{\alpha_1}(\lambda)$, 
$\mu_i = s_{\beta_{i-1}}\cdots s_{\beta_1}(\mu)$. 
Let $R$ be the quotient of $\tilde{R}$ by right polynomial torsion.
Then $R$ has no polynomial torsion.

\begin{proposition}
Suppose $(X,Y,R, \check{R}, \Pi)$ is a simply connected semisimple root datum.
Let $H$ be any datum satisfying the conditions of \ref{conditions} and which is included in the image $k[\dT]\hookrightarrow k[\dT]_\lambda$, so that the algebras $\grhecke, R$ are defined.
Then the map on generators $1_\lambda\mapsto 1_\lambda$, $x_\alpha^\lambda\mapsto P_\alpha - \langle \lambda , \alpha\rangle$, $\tau_\alpha^\lambda\mapsto r_\alpha^\lambda$ is an isomorphism of algebras.

Moreover, suppose for each $\alpha\in \pi, \lambda\in \dT$ that $\psi_\lambda(H_\alpha^\lambda)$ is a homogeneous polynomial in $\{x_\alpha\}_{\alpha\in \Pi}$.
Then let $\operatorname{deg}$ be defined on generators as $\operatorname{deg}(1_\lambda) = 0$, $\operatorname{deg}(x_\alpha^\lambda) = 2$, and 
\begin{align*}
\operatorname{deg}(\tau_\alpha^\lambda) = \begin{cases}
-2 & \text{if $s_\alpha(\lambda) = \lambda$,} \\
\frac{1}{2} \operatorname{deg}(\psi_\lambda(H_\alpha^\lambda)) & \text{else.}
\end{cases}
\end{align*}
Then $\operatorname{deg}$ extends to a grading on the algebra $R$, hence on the quiver Hecke algebra $\grhecke$.
\end{proposition}

\begin{corollary}
Let $H$ be the data associated above to a degenerate affine Hecke algebra as above.
Then the quadratic relations for $R$ are given by,
\begin{gather*}
\tau_\alpha^{s_\alpha(\lambda)}\tau_\alpha^\lambda = \begin{cases}
0 & \text{if }\langle\alpha ,  \lambda\rangle = 0, \\
x_\alpha^\lambda & \text{if }\langle \alpha, \lambda\rangle = c_\alpha, \\
-x_\alpha^\lambda & \text{if }\langle \alpha, \lambda\rangle = -c_\alpha \\
1_\lambda & \text{else},
\end{cases},
\end{gather*}
whereas, for characteristic 2 let,
\begin{gather*}
\tau_\alpha^{s_\alpha(\lambda)}\tau_\alpha^\lambda = \begin{cases}
0 & \text{if }\langle\alpha ,  \lambda\rangle = 0, \\
(x_\alpha^\lambda)^2 & \text{if }\langle \alpha, \lambda\rangle = c_\alpha, \\
1_\lambda & \text{else},
\end{cases},
\end{gather*}

The grading is as follows: $\operatorname{deg}(1_\lambda) = 0$, $\operatorname{deg}(x_\alpha^\lambda) = 2$, for characterstic of $k$ not $2$,
\begin{align*}
\operatorname{deg}(\tau_\alpha^\lambda) = \begin{cases}
-2 & \text{if $s_\alpha(\lambda) = \lambda$,} \\
1 & \text{if $\langle \lambda, \alpha\rangle = \pm c_\alpha$,} \\
0 & \text{else,} 
\end{cases}
\end{align*}
whereas for characteristic of $k$ equal to $2$,
\begin{align*}
\operatorname{deg}(\tau_\alpha^\lambda) = \begin{cases}
-2 & \text{if $s_\alpha(\lambda) = \lambda$,} \\
2 & \text{if $\langle \lambda, \alpha\rangle = c_\alpha$,} \\
0 & \text{else.} 
\end{cases}
\end{align*}
\end{corollary}

\begin{proof}
The presentation of the algebra $R$ is simply a recollection of the relations of $\qhecke$, restricted to $\grhecke$ and with the above change of variables.

To show that the above grading on generators extends to a grading on $\grhecke$ we need to show that all the relations are graded.
This is trivial with the exception of the braid relation and the relations pertaining to polynomial torsion.
In the case of the braid relation, we have
\begin{align*}
\langle \lambda_i , \alpha_i\rangle = \langle \lambda , s_{\alpha_1}\cdots s_{\alpha_{i-1}}(\alpha_i) \rangle.
\end{align*}
Applying lemma \ref{permute} to the expression $s_{\alpha_m}\cdots s_{\alpha_1} = s_{\beta_m}\cdots s_{\beta_1}$ we find a permutation $p$ of the set $\{1, \cdots , m\}$ with,
\begin{align*}
\langle \lambda_i , \alpha_i\rangle =
\langle \lambda , s_{\alpha_1}\cdots s_{\alpha_{i-1}}(\alpha_i) \rangle = 
\langle \lambda , s_{\beta_1}\cdots s_{\alpha_{p(i)-1}}(\beta_{p(i)}) \rangle = 
\langle \mu_{p(i)} , \beta_{p(i)} \rangle
\end{align*}
It follows that the braid relations are graded.

As for the polynomial torsion, section \ref{PBW-section} shows that the polynomial torsion relation may be replaced with the following relations for general $\lambda\in \dT$ and $\alpha,\beta\in \Pi$ distinct.
For $\lambda$ not fixed by either $s_\alpha, s_\beta$, then we may pick $t$ as in section \ref{PBW-section}. 
In that case, following the notation of \ref{PBW-section}, we have the following braid-like relation:
\begin{align*}
\tau_{\alpha_m}^{\lambda_m}\cdots \tau_{\alpha_1}^{\lambda_1}  - 
\tau_{\beta_m}^{\mu_m}\cdots \tau_{\beta_1}^{\mu_1} = 
s_{\alpha_m}\cdots s_{\alpha_{t+2}} (\bgg_{\alpha_{t+1}}(P)) \tau_{\alpha_m}\cdots \tau_{\alpha_{2t+2}},
\end{align*}
where,
\begin{align*}
P =& \prod_{i=1}^t s_{\alpha_t}\cdots s_{\alpha_{i+1}}(G_{\alpha_i}^{\lambda_{-i}}),\\
=& \tau_{\alpha_t}^{\lambda_t}\cdots \tau_{\alpha_1}^{\lambda_1}
\tau_{\alpha_1}^{\lambda_{-1}}\cdots \tau_{\alpha_t}^{\lambda_{-t}}.
\end{align*}
Here, $\lambda_{-i} = s_{\alpha_i}(\lambda_i)$.
The two terms on the left side of the relation have the same degree by the above description of $\langle \lambda_i , \alpha_i\rangle$.
The term $\Delta_{\alpha_{t+1}}(P)$ has the same degree as $\tau_{\alpha_{2t+1}}^{\lambda_{2t+1}}\cdots \tau_{\alpha_1}$ as the product expression for $P$ shows that it has the same terms with the exception of $\tau_{\alpha_t}^{\lambda_t}$, which has degree $-2$.
As the operator $\Delta_\alpha$ has degree $-2$, it follows that the above braid-like relation is graded.

\end{proof}

\subsection{Graded characters of irreducible representations}
Let $\grhecke$ be the quiver Hecke algebra with grading defined above.
We define an anti-involution $\iota : \grhecke\to \grhecke^{opp}$ as follows:
\begin{align*}
\iota(1_\lambda) =& 1_\lambda \\
\iota(x_\alpha^\lambda) =& x_\alpha^\lambda \\
\iota(\tau_\alpha^\lambda) =& \tau_\alpha^{s_\alpha(\lambda)}.
\end{align*}
The only difficulty in showing that $\iota$ is well defined is in showing that the braid relations are preserved under $\iota$.
Indeed, to show that 
\begin{align*}
\tau_{\alpha_1}^{\lambda_2}\cdots \tau_{\alpha_m}^{\lambda_{m+1}} - \tau_{\beta_1}^{\mu_2}\cdots \tau_{\beta_m}^{\mu_{m+1}} =  
s_{\alpha_{2t+1}}\cdots s_{\alpha_{t+2}}\bgg_{\alpha_{t+1}}(P)
\tau_{2t+2}^{\lambda_{2t+3}}\cdots r_{\alpha_m}^{\lambda_{m+1}},
\end{align*}
it suffices to multiply the left side of the equation on the left by $\tau_{\alpha_1}\cdots \tau_{\alpha_t} \tau_{\alpha_t} \cdots \tau_{\alpha_1}$ and simplify as in section \ref{PBW-section}.

As $\operatorname{deg}(\tau_\alpha^\lambda) = \operatorname{deg}(\tau_\alpha^{s_\alpha(\lambda)})$, we see that $\iota$ is a graded anti-involution.
Let $V$ be a graded $\grhecke$-module.
Then $V^\vee:=\hom_{k}(V,k)$ is naturally a $\grhecke^{opp}$-module, which we consider as a $\grhecke$-module via the map $\iota$.
We see that the graded character of $V^\vee$ is given by switching $v$ and $v^{-1}$ in the graded character of $V$.

\begin{proposition}
Let $\grhecke$ be the quiver Hecke algebra associated to a degenerate affine Hecke algebra as above.
Let $V$ be an irreducible graded representation of $\grhecke$.
There is a grading shift $V\{\ell\}$ of $V$ such that the graded character of $V\{\ell\}$ is invariant under the substitution $v\mapsto v^{-1}$.
\end{proposition}
\begin{proof}
By Schur's lemma, a non-graded irreducible $\grhecke$-module can have at most one grading, up to grading shift.
As an irreducible $\grhecke$-module is determined up to isomorphism by its ungraded character we find that $V$ and $V^\vee$ are isomorphic as ungraded modules.
It follows that the graded character of $V$ is $v^k$ times the graded character of $V^\vee$ for some $k$.
First, we claim that $k$ must be even.
It suffices to consider the proposition for the graded character of $V_\lambda$ as an $1_\lambda \grhecke 1_\lambda$-module, where $V_\lambda$ is some non-zero generalized eigenspace.
Let $v\in V_\lambda$ be an element of highest degree.
As $\operatorname{deg}(x_\alpha^\lambda) >0$, we see that $x_\alpha^\lambda v = 0$ for all $\alpha\in \Pi$.
We have $V_\lambda = 1_\lambda \grhecke 1_\lambda v$, so by the PBW-theorem for $\grhecke$ we find that $V_\lambda$ is the $k$-span over elements $\tau_{\alpha_m}^{\lambda_m} \cdots \tau_{\alpha_1}^{\lambda_1}$ with $\lambda = \lambda_1 = s_{\alpha_m}(\lambda_m)$, and the other $\lambda_i$ defined as usual by $\lambda_{i+1} = s_{\alpha_i}(\lambda_i)$.
We claim that the degree of such an element $\tau_{\alpha_m}^{\lambda_m} \cdots                  \tau_{\alpha_1}^{\lambda_1}$ must be even.
In characteristic $2$ this is automatic, as every $\tau_\alpha^\lambda$ has even degree.
Otherwise, the claim is equivalent to showing that the number of $i$ for which $\langle \lambda_i , s_{\alpha_i}\rangle = \pm c_{\alpha_i}$ is even.
There exists $u\in W$ with $\lambda' = u^{-1}(\lambda)$ a standard parabolic weight.
Given a decomposition $u = s_{\beta_n}\cdots s_{\beta_1}$, we easily see that the element, 
\begin{align*}
\tau_{\beta_1}\cdots\tau_{\beta_n}^\lambda\tau_{\alpha_k}^{\lambda_k}\cdots\tau_{\alpha_1}^{\lambda}\tau_{\beta_n}\cdots\tau_{\beta_1}^{\lambda'},
\end{align*}
has the same degree as $\tau_{\alpha_m}^{\lambda_m} \cdots \tau_{\alpha_1}^{\lambda_1}$ modulo 2, as the former element has for every $\tau_{\beta_i}^{\mu}$ term, a term of the form $\tau_{\beta_i}^{s_{\beta_i}(\mu)}$, and these two degrees add up to 2.

Thus, we are left to show that for $\lambda$ a parabolic weight and $s_{\alpha_m}\cdots s_{\alpha_1}(\lambda) = \lambda$ we must have the number of $i$ with $\langle \lambda_i , \alpha_i\rangle = \pm c_{\alpha_i}$ is even. 
If $s_{\alpha_m}\cdots s_{\alpha_1}$ is a reduced expression then each $\lambda_i = \lambda$, so that $\langle \lambda_i , \alpha_i \rangle = 0$ and the claim follows.

Note that 
\begin{align*}
\langle \lambda_i , \alpha_i\rangle = \langle \lambda, s_{\alpha_1}\cdots s_{\alpha_{i-1}}(\alpha_i)\rangle .
\end{align*}
If $s_{\alpha_m}\cdots s_{\alpha_1}$ is not reduced, then there exists $i<j$ with 
\begin{align*}
s_{\alpha_j}\cdots s_{\alpha_{i}} =& s_{\alpha_{j-1}}\cdots s_{\alpha_{i+1}} , \\
s_{\alpha_j}\cdots s_{\alpha_{i+1}} =& s_{\alpha_{j-1}}\cdots s_{\alpha_{i}},
\end{align*}
both of which are reduced expressions.
It follows for this choice of $i,j$ that,
\begin{align*}
\alpha_i = s_{\alpha_{i}}\cdots s_{\alpha_{j-1}}(\alpha_j),
\end{align*}
and hence,
\begin{align*}
\langle \lambda_i , \alpha_i \rangle =& 
\langle \lambda s_{\alpha_1}\cdots s_{\alpha_{i-1}}(\alpha_i) \rangle \\ 
=& 
\langle \lambda s_{\alpha_1}\cdots s_{\alpha_{j-1}}(\alpha_j) \rangle \\
= &
\langle \lambda_j , \alpha_j \rangle.
\end{align*}
Moreover, as the equation $s_{\alpha_j}\cdots s_{\alpha_{i+1}} = s_{\alpha_{j-1}}\cdots s_{\alpha_i}$ is an equality of reduced expressions, there is a bijection $p:\{i+1,\cdots j\}\xrightarrow{\sim} \{i,\cdots , j-1\}$ with the property that,
\begin{align*}
s_{\alpha_{i+1}}\cdots s_{\alpha_{k-1}}(\alpha_k) =& s_{i}\cdots s_{\alpha_{p(k)-1}}(\alpha_{p(k)}),\\
p(j) =& i.
\end{align*}
It follows that modulo 2, the degree of $\tau_{\alpha_m}^{\lambda_m}\cdots \tau_{\alpha_1}^{\lambda_1}$ is the same as the degree of the $\tau_{\beta_{m-2}}^{\mu_{m-2}}\cdots \tau_{\beta_1}^{\mu_1}$ where the sequence $\beta_1, \cdots , \beta_{m-2}$ is the same as the sequence $\alpha_1 , \cdots \hat{\alpha_i} , \cdots , \hat{\alpha_j}, \cdots , \alpha_m$, where the hat denotes ommision, and the $\mu_i$ are defined as usual, $\mu_1 = \lambda$, $\mu_{i+1} = s_{\beta_i}(\mu_i)$.

By induction it follows that the degree of any sequence $\tau_{\alpha_m}^{\lambda_m}\cdots \tau_{\alpha_1}^{\lambda_1}$ with $\lambda_1 = s_{\alpha_m}(\lambda_m)$ is even.
We have proved that for any irreducible graded $\grhecke$-module $V$ with $V_\lambda\not = 0$, the degree of any two elements in $V_\lambda$ differ by a multiple of two.
In particular the highest degree and lowest degree element in the graded character of $V_\lambda$ differ by $v^k$ where $k$ is even.
Put $\ell = \frac{k}{2}$.
We then have that $V\{-\ell\}$ is a module whose graded character is invariant under the substitution $v\mapsto v^{-1}$.
\end{proof}

\begin{remark}
As one can see, the result is true more generally if $\operatorname{deg}(\tau_\alpha^\lambda)$ depends only on $\langle \lambda, \alpha \rangle$. 
\end{remark}
\begin{corollary}
The category of finite representations of a degenerate affine Hecke algebra $\dhecke$ associated to a simply connected semi-simple root data is Morita equivalent to the category of ungraded finite representations of the associated quiver Hecke algebra $\grhecke$ for which the elements $x_\alpha^\lambda$ act nilpotently for all $\alpha\in \Pi, \lambda\in \dT$.
An irreducible representation $V$ of $\dhecke$ is associated to a unique graded irreducible representation of $\grhecke$, whose graded character is invariant under the substitution $v\mapsto v^{-1}$, and for which the substitution $v\mapsto 1$ yields the character of $V$.
\end{corollary}
\begin{proof}
We must show that every ungraded irreducible representation of $\grhecke$ has a grading compatible with the action of $\grhecke$.
As the center $Z(\grhecke)$ of $\grhecke$ is given by the space of invariants $(\bigoplus_{\lambda\in \dT}k[\dT]1_\lambda)^W$, it is a graded ideal of $\grhecke$. 
Moreover, the action of $\grhecke$ on an irreducible representation factors through the finite dimensional graded algebra $\grhecke / Z(\grhecke)$.
Then the claim follows from \cite[Theorem 4.4.4]{nv}.
\end{proof}

\subsection{A pair of adjoint functors} \label{weight-induction}

Fix $\lambda\in \iT$ and consider the algebra $\whecke:= 1_\lambda \lhecke 1_\lambda$, which we refer to as the \emph{weight Hecke algebra}.
There is a functor on finite dimensional representations which we will call the 
$\lambda$-weight restriction functor,
\begin{eqnarray*}
wRes_\lambda :  \ihecke-\mod \longrightarrow & 
\whecke-\mod, \\
V \mapsto & V_\lambda.
\end{eqnarray*}

The weight restriction functor admits a left adjoint, which we will call $wInd_\lambda$, or
the $\lambda$-weight induction functor,
\begin{eqnarray*}
wInd_\lambda : \whecke-\mod \longrightarrow & 
\ihecke-\mod, \\
V_\lambda \longrightarrow& \lhecke1_\lambda \otimes_{\whecke} V_\lambda.
\end{eqnarray*}

\begin{proposition}\label{irreducible-reps}
Let $\lambda\in \iT$ be a weight.
The following gives a construction of all irreducible representations $V$ of $\ihecke$ for which
$V_\lambda\not=0$, in terms
of irreducible representations of $\whecke$.
\begin{enumerate}
\item
Let $V_\lambda$ be a non-zero, irreducible $\whecke$-module.
Then,
$wInd_\lambda(V_\lambda)$ has a unique irreducible quotient, $L(wInd_\lambda(V_\lambda))$.
\item
Conversely, suppose that $V$ is an irreducible representation of $\ihecke$ with $V_\lambda\not=0$.
Then,
$V_\lambda$ is an irreducible ${}_\lambda \mathscr{H}_\lambda$-module, and the counit of the adjunction,
non-zero map
\[
wInd_\lambda(V_\lambda)\to V,
\]
is non-zero and identifies $V$ with $L(wInd_\lambda(V_\lambda))$.
\item
Finally, the kernel of the above map is the largest $\ihecke$-submodule $U$ 
of $wInd_\lambda(V_\lambda)$
for which $U_\lambda=0$. 
This kernel may be computed in terms of 
the $\whecke$-module structure on $V_\lambda$.
\end{enumerate}
\end{proposition}

\begin{proof}
For the first claim, suppose $V_\lambda$ is an irreducible $\whecke$-module and let $U\subsetneq wInd_\lambda(V_\lambda)$ be a proper $\whecke$-submodule. 
We claim that $U_\lambda=0$. 
If not, $V_\lambda$ irreducible implies that $U_\lambda = V_\lambda$. But, $V_\lambda$ generates 
$\lhecke1_\lambda\otimes_{\whecke} V_\lambda$ as an
$\ihecke$-module, so we would have $U=wInd_\lambda(V_\lambda)$, a contradiction.

Now we note that the interior sum of two submodules $U,U'\subset wInd_\lambda(V_\lambda)$
with $U_\lambda = U'_\lambda=0$ is a submodule, $U+U'$, with $(U+U')_\lambda=0$.
So there is a unique maximal proper submodule, categorized as the sum of 
all $\ihecke$-submodules $U$ with $U_\lambda=0$.

For the second claim, we have the easy fact, $wInd_\lambda(V_\lambda)_\lambda = 1_\lambda \lhecke 1_\lambda \otimes_{\whecke}V_\lambda \cong V_\lambda$.
If there were
a non-trivial submodule $U_\lambda \subset V_\lambda$, then the image $U'$ of the composition of maps
\[
wInd_\lambda(U)\to wInd_\lambda V_\lambda \to V
\]
would be a submodule of $V$ with $0\subsetneq U'_\lambda \subsetneq V_\lambda$. 
Thus, 
$U'$ would be a non-trivial $\ihecke$-submodule of $V$.

The last claim follows from the proof of the first claim. We show how
to describe the maximal proper submodule.

Let $U_{\lambda'}$ be the left $\iA$-span of the elements of the form
$r\otimes v\in 1_{\lambda'}\lhecke1_\lambda\otimes V_\lambda$ for which
$1_\lambda \lhecke1_{\lambda'} (r\otimes v) = 0$. Then $U_{\lambda'}$
is clearly the $\lambda'$-weight space of the maximal proper submodule.
We may describe this set as $1_{\lambda'}\lhecke1_\lambda\otimes V_\lambda^{\lambda'}$,
where $V_\lambda^{\lambda'}$ is the kernel of the action of 
$1_\lambda \lhecke1_{\lambda'}\lhecke1_\lambda\subset 
\whecke$ on $V_\lambda$.
\end{proof}

\subsection{The Demazure algebra}\label{demazure-algebra}
For simply connected root datum, $(X,Y,R,\check{R},\Pi)$,
we may identify the algebra, $End_{(\iA)^W}(\iA)$ 
with the Demazure algebra, $\nhecke$, an interpolating version of the affine nil-Hecke algebra ${}^0H$ of \cite[Section 2.1]{rouquier-qha}. As a vector space
this algebra is equal to a tensor product
\begin{align*}
\iA\otimes_\C \nhecke^f
\end{align*}
of $\iA$ with the finite Demazure algebra, $\nhecke^f$ of $W$.
the latter algebra is the algebra with generators $\tau_\alpha, \alpha\in\Pi$,
which satisfy the braid relation between $\tau_\alpha,\tau_\beta$, as well as the quadratic relation
\begin{align*}
\tau_\alpha^2 = h_0 \tau_\alpha.
\end{align*}
The algebra structure of $\nhecke$ is given by letting $\iA$ and $\nhecke^f$
be subalgebras, and giving the commutativity relation,
\begin{align*}
\tau_\alpha f- s_\alpha(f) \tau_\alpha = h\demazure_\alpha(f).
\end{align*}

The following theorem is an algebraic link between the Demazure algebra
and the weight Hecke algebra, $\whecke$,
for certain $\lambda\in \iT$.

It is clear that for a weight, $\lambda\in\iT $, the subalgebra
of $\whecke$ generated by $\iA_\lambda$ and
$r_\alpha^\lambda$ with $s_\alpha(\lambda)=\lambda$ is isomorphic to a Demazure algebra with possibly smaller root datum, $(X,Y,R^\lambda,\check{R}^\lambda,\Pi^\lambda)$, 
$\Pi^\lambda = \{\alpha\in \Pi\mid s_\alpha(\lambda) = \lambda\}$.
There is a special case when this subalgebra is the entirety of $1_\lambda \mathscr{H}1_\lambda$.

We can use simple Weyl group lemmas and the structure theorem of 
the quiver Hecke algebra $\mathscr{H}(G)$ to give the solution to this question.

\begin{theorem}
Suppose $\lambda\in \iT$ is a standard parabolic weight, \emph{i.e.}
the stabilizer of $\lambda$ in $W$ is a standard parabolic subgroup of $W$. 
Then the weight-Hecke algebra $\whecke$ is isomorphic
to the Demazure algebra associated to the root data,
$(X,Y,R^\lambda,\check{R}^\lambda, \Pi^\lambda)$.
\end{theorem}
\begin{corollary}
If $\lambda\in \iT$ is a parabolic weight, then there is, up to isomorphism,
only one irreducible representation $V$ of $\ihecke$ with $V_\lambda\not=0$.
\end{corollary}
\begin{proof}
By corollary \ref{structure}, we see that $\whecke$ is spanned
by products $r_{\alpha_n}\cdots r_{\alpha_1}$ with 
$w = s_{\alpha_n}\cdots s_{\alpha_1}$ a reduced expression for $w\in W$,
a Weyl group element which stabilizes $\lambda$. By assumption
the stabilizer is generated by $s_\alpha,\alpha\in \Pi$ fixing $\lambda$,
and a reduced expression will use only these terms $s_\alpha, \alpha\in \Pi^\lambda$.

Now, for $(X,Y,R,\check{R},\Pi)$ simply connected, a parabolic subgroup
corresponding to $\Pi^\lambda$ will also be simply connected. Thus, 
the subalgebra $\whecke\cong \nhecke$ will be a matrix algebra
over $\mathscr{A}^{W^\lambda}$. Modulo the kernel of the central character corresponding to the irreducible representation, the algebra is a matrix algebra over $\C$. 
Thus, the weight Hecke algebra $\whecke$ has only one irreducible representation
with a non-zero weight $\lambda$. In fact, it's dimension is $\#W^\lambda$, the cardinality of the stabilizer of $\lambda$.
\end{proof}

\subsection{Example computation}
Let us take $(X,Y,R,\check{R}, \Pi)$ the standard root datum for $\operatorname{SL}_3$, and $\grhecke$ the quiver Hecke algebra with the grading given above.
Let $v$ stand for the grading shift, so that characters of finite $\grhecke$ modules are in the group ring $\Z[v^{\pm 1}][\dT]$.
Let $k$ be a field with characteristic not 2.
We have the roots, $\alpha = (1,-1,0),
\beta = (0,1,-1)$, which span the vector space $X\otimes_\Z k$.
Then $\mathbb{A} = S_k(X)$, the symmetric algebra is a polynomial algebra
in the variables $\alpha, \beta$. 
Pick a parabolic weight $\lambda = (1,1,0)\in \Hom_{alg}(\mathbb{A}, \C)$ and let $\Lambda$ be the $\mathfrak{S}_2$-orbit of $\lambda$. 
We compute the graded characters of each irreducible representation of $\grhecke$ whose associated representation of $\dhecke$ have central character $\Lambda$.

First, suppose $V$ is a finite irreducible $\grhecke$-module with non-zero $\lambda$-weight space.
Since $\lambda$ is a standard parabolic weight, there is only one such representation up to isomorphism and we may construct it as follows.
By the previous section, $1_\lambda \grhecke 1_\lambda$ is isomorphic to the nil-affine Hecke algebra for the root system $(X, Y, \{\alpha\}, \{\check{\alpha} , \{\alpha\})$. 
There are two ways of constructing the irreducible $1_\lambda \grhecke 1_\lambda$-module, $V_\lambda$.
Let $k$ be the trivial $\dA$-module with $x_\alpha,x_\beta$ acting by 0.
Then $1_\lambda \grhecke 1_\lambda \otimes_{\dA 1_\lambda} k$ has the correct dimension, and hence must be the unique irreducible representation of $1_\lambda \grhecke 1_\lambda$ with $x_\alpha,x_\beta$ acting nilpotently.

Alternatively we could induce from the finite nil Hecke algebra.
Let $\dA^{\langle s_\alpha\rangle}$ be the functions which are invariant under the action of $s_\alpha$.
Let $J_0$ be the positively graded elements of this subalgebra.
Then $J_0$ is a central ideal of $1_\lambda \grhecke 1_\lambda$, and we can form the representation $1_\lambda \grhecke 1_\lambda \otimes_{{}^0\ahecke^f}k$, where $k$ is the trivial nil-Hecke ${}^0\ahecke^f$-module with $\tau_\alpha^\lambda$ acting by 0.
Again, this representation has the correct dimension and so must be isomorphic to the unique irreducible nilpotent representation.

It is clear that the graded character of the first representation is $1+v^2$, whereas the graded character of the second one is $v^{-2}+1$.
We may shift the grading so that the graded character of this module is $v+v^{-1}$, which is invariant under the substitution $v\mapsto v^{-1}$.
Let $L(V_\lambda)$ be the irreducible quotient of the module weight-induced from $V_\lambda$.
Then the graded character of $L(V_\lambda)_{s_\beta(\lambda)}$ is simply $1$, and the character of $L(V_\lambda)_{s_\alpha s_\beta(\lambda)}$ is simply 0.
It is worth noting that the graded character of $L(V_\lambda)$ is invariant under the substitution $v\mapsto v^{-1}$.

The case of irreducible representations with non-zero $s_\alpha s_\beta(\lambda)$-weight space is identical as that is also a standard parabolic weight.
We are left to compute the irreducible modules $V_{s_\beta(\lambda)}$ over $1_{s_\beta(\lambda)}\grhecke 1_{s_\beta(\lambda)}$ whose associated irreducible $\grhecke$-module $L(V_{s_\alpha})$ has no $\lambda, s_\alpha s_\beta(\lambda)$-weight space.
This is equivalent to $\tau_\beta^\lambda\tau_\beta^{s_\beta(\lambda)} = x_\beta$, and $\tau_\alpha^{s_\alpha s_\beta(\lambda)}\tau_\alpha^{s_\beta(\lambda)} = -x_\alpha$ acting by 0 on $V_{s_\beta(\lambda)}$.
By the commutativity relation between $\tau_\gamma: = \tau_\alpha \tau_\beta \tau_\alpha^{s_\beta(\lambda)}$ and $x_\gamma := x_{\alpha+\beta}$, $\tau_\gamma x_{\gamma}+x_\gamma \tau_\gamma = 2$, we see that the constant map $2$ must be zero on $V_{s_\beta(\lambda)}$.
It follows that there is no such non-zero irreducible representation.

\subsection{Appendix: quiver Hecke algebras in type $A$}
We note that in type $A$ the (graded) quiver Hecke algebras appearing in \cite{rouquier-qha} are related to 
the algebra $\qhecke$ we have defined here with the standard root datum $GL_n$.
We first present the quiver Hecke algebra $\rhecke$ from \cite{rouquier-qha} which is shown there to be related to the degenerate affine Hecke algebras of type $A$.
\begin{definition}
Define the quiver $\Gamma$ with vertices $I=k$ and arrows $a\to a+1$, $a\in k$. 
Denote $ \lambda = (\lambda_1,\lambda_2,\dots , \lambda_n)\in k^n$ with $\lambda_i\in I, i=1,2,\dots ,n$.
Then $\rhecke$ is the algebra generated by idempotents $\{1_\lambda\}_{\lambda\in \Lambda}$, by variables $\{x_i^\lambda\}_{i=1}^n$, and $\{\tau_i^\lambda\}_{i\in I}$ subject to the following relations.

\begin{gather*}
1_\lambda 1_{\lambda'} = 1_\lambda \delta_{\lambda,\lambda'}, \\
x_i^\lambda 1_{\lambda'} = 1_{\lambda'} x_i^\lambda = \delta_{\lambda,\lambda'}x_i^\lambda, \\
\tau_i^\lambda 1_{\lambda'} = 1_{s_i(\lambda')}\tau_i^\lambda = \delta_{\lambda,\lambda'}\tau_i^\lambda, \\
x_i^\lambda x_j^\lambda = x_j^\lambda x_i^\lambda, \\
\tau_i x_j^\lambda - x_{s_i(j)}^{s_i(\lambda)} \tau_i^\lambda = \begin{cases}
-1_\lambda & \text{if $s_i(\lambda) = \lambda$ and $j=i$,} \\
1_\lambda & \text{if $s_i(\lambda) = \lambda$ and $j=i+1$,} \\
0 & \text{else},
\end{cases} \\
\tau_i^{s_i(\lambda)}\tau_i^\lambda = \begin{cases}
0 & \text{if }\lambda_i = \lambda_j \\
x_{i+1}^\lambda - x_{i}^\lambda & \text{if }\lambda_j = \lambda_i + 1 \\
x_{i}^\lambda - x_{i+1}^\lambda & \text{if }\lambda_j = \lambda_i - 1 \\
1_\lambda & \text{else},
\end{cases} \\
\tau_i^{s_j(\lambda)}\tau_j^\lambda - \tau_j^{s_i(\lambda)} \tau_i^\lambda = 0 \text{ if $\mid i-j \mid >1$,} \\
\tau_i^{s_{i+1}s_i(\lambda)}\tau_{i+1}^{s_i(\lambda)}\tau_i^\lambda - 
\tau_{i+1}^{s_i s_{i+1}(\lambda)}\tau_i^{s_{i+1}(\lambda)}\tau_{i+1}^\lambda = 0
\text{ if $\lambda_{i} = \lambda_{i+1} = \lambda_{i+2}$ or $\lambda_{i}\not = \lambda_{i+2}$}.
\end{gather*}

There is also a relation stating that $\rhecke$ have no polynomial torsion (see \cite[Section 3.2.2]{rouquier-qha}), which is equivalent to the missing braid like relations from \cite[Section 3.2.1]{rouquier-2km}.
\end{definition}
Let $(\Z^n, \Z^n, R,\check{R},\Pi)$ be a root data for $GL_n$ with standard basis $\{e_i\}_{i=1}^n$ of $\Z^n$ and simple roots $\alpha_i = e_{i+1}-e_i$, $i=1,\dots , n-1$. 
Let $c_\alpha = 1$ be a set of parameters, and let $\dhecke$ be the associated degenerate affine Hecke algebra, with $T_i = T_{\alpha_i}$.
We remark that although the weight lattice for $GL_n$ does not contain the fundamental weights, it is the case that every weight is conjugate to a standard parabolic weight.
Thus, our construction of $\qhecke, \grhecke$ may be carried out with no change.

Finally, let $V$ be a representation of the degenerate affine Hecke algebra $\dhecke$. 
In this case we identify $P_i = P_{e_i}$, $\dA = k[P_1, \dots , P_n]$, and $\lambda\in (k^*)^n$ a weight of $\dA$ via $\lambda(P_i) = \lambda_i$. 
By \cite[Theorem 3.11]{rouquier-2km}, we can turn $V$ into an $\rhecke$-module with $1_\lambda$ the projection onto the $\lambda$-generalized-eigenspace for $\dA$, $x_i^\lambda$ acting by $(P_i - \lambda_i)1_\lambda$, and $\tau_i^\lambda$ acting by,
\begin{align*}
\tau_i^\lambda \mapsto \begin{cases}
(P_i - P_{i+1}+1)^{-1} (T_i - 1)1_\lambda & \text{if $s_i(\lambda) = \lambda$,} \\
((P_i - P_{i+1})T_i + 1)1_\lambda & \text{if $\lambda_{i+1} = \lambda_{i}+1$,}  \\
\frac{P_i - P_{i+1}}{P_i - P_{i+1}+1} (T_i-1)1_\lambda + 1_\lambda & \text{else.}
\end{cases}
\end{align*}
Using our notation for $\alpha_i$, as well as the relation commutativity relation for $T_i$ and $1_\lambda$, we find that this is equivalent to,
\begin{align*}
\tau_i^\lambda \mapsto \begin{cases}
(1-P_{\alpha_i})^{-1} (T_i - 1)1_\lambda & \text{if $s_i(\lambda) = \lambda$,} \\
(-P_{\alpha_i}) 1_{s_i(\lambda)}T_i 1_\lambda & \text{if $\lambda_{i+1} = \lambda_{i}+1$,}  \\
(\frac{-P_{\alpha_i}}{1-P_{\alpha_i}}) 1_{s_i(\lambda)}(T_i)1_\lambda & \text{else.}
\end{cases}
\end{align*}

Using the isomorphism $\qhecke \to \lhecke$ of the previous section we find the following:
\begin{theorem}
There is a map $\rhecke\to \qhecke$ given by $x_i^\lambda\mapsto (X_i - \lambda_i) 1_\lambda\in \lA$, and 
\begin{align*}
\tau_i^\lambda\mapsto \begin{cases}
r_{\alpha_i}^\lambda & \text{if $s_i(\lambda)=\lambda$,} \\
(-\alpha_i) r_{\alpha_i}^\lambda & \text{if $\lambda_{i+1} = \lambda_i+1$} \\
(\frac{-\alpha_i}{1-\alpha_i}) r_{\alpha_i}^\lambda & \text{else.}
\end{cases}
\end{align*}
Moreover, this map gives a graded isomorphism of $\rhecke$ onto $\grhecke$.
\end{theorem}

We can do something analogous for localized affine Hecke algebras, but only in the simply laced case.

\begin{definition}
Let $q_\alpha = q\in k^*$.
Define the quiver $\Gamma$ of \cite[Section 3.2.5]{rouquier-2km} with vertices $I=k^*$ and arrows $a\to q\cdot a$.
Denote $\lambda = (\lambda_1 , \dots , \lambda_n)\in (k^*)^n$.
Then $\rhecke$ is the algebra generated by idempotents $\{1_\lambda\}_{\lambda\in (k^*)^n}$, variables $\{x_i^\lambda\}_{i=1}^{n}$ and $\{\tau_i^\lambda\}_{i=1}^n$ with relations:
\begin{gather*}
1_\lambda 1_{\lambda'} = \delta_{\lambda,\lambda'} 1_\lambda, \\
x_i^\lambda 1_{\lambda'} = 1_{\lambda'}x_i^\lambda = \delta_{\lambda,\lambda'}x_i^\lambda, \\
\tau_i^\lambda 1_{\lambda'} = 1_{s_i(\lambda')}\tau_i^\lambda = \delta_{\lambda,\lambda'}\tau_i^\lambda, \\
x_i^\lambda x_j^\lambda = x_j^\lambda x_i^\lambda, \\
\tau_i^\lambda x_j^\lambda - x_{s_i(j)}^{s_i(\lambda)} \tau_i^\lambda = \begin{cases}
-1_\lambda & \text{if $s_i(\lambda) = \lambda$ and $j=i$,} \\
1_\lambda & \text{if $s_i(\lambda) = \lambda$ and $j=i+1$,} \\
0 & \text{else},
\end{cases} \\
\tau_i^{s_i(\lambda)}\tau_i^\lambda = \begin{cases}
0 & \text{if }\lambda_i = \lambda_j \\
x_{i+1}^\lambda - x_{i}^\lambda & \text{if }\lambda_j = \lambda_i \cdot q \\
x_{i}^\lambda - x_{i+1}^\lambda & \text{if }\lambda_j = \lambda_i \cdot q^{-1} \\
1_\lambda & \text{else},
\end{cases} \\
\tau_i^{s_j(\lambda)}\tau_j^\lambda - \tau_j^{s_i(\lambda)} \tau_i^\lambda = 0 \text{ if $\mid i-j \mid >1$,} \\
\tau_i^{s_{i+1}s_i(\lambda)}\tau_{i+1}^{s_i(\lambda)}\tau_i^\lambda - 
\tau_{i+1}^{s_i s_{i+1}(\lambda)}\tau_i^{s_{i+1}(\lambda)}\tau_{i+1}^\lambda = 0
\text{ if $\lambda_{i} = \lambda_{i+1} = \lambda_{i+2}$ or $\lambda_{i}\not = \lambda_{i+2}$}.
\end{gather*}
Again, there is also a relation (see \cite{rouquier-qha}) in $\rhecke$ saying that $\rhecke$ contains no polynomial torsion which accounts for the missing braid like relations in \cite{rouquier-2km}.
\end{definition}

Let $c_\alpha = 1, q_\alpha = q\in k^*\backslash \{\pm 1\}$ and $h_0 = q-1$.
Again, let $(\Z^n, \Z^n, R, \check{R},\Pi)$ be a root datum for $GL_n$.
Let $\{e_i\}_{i=1}^n$ be the standard basis of $\Z^n$, and let the root basis be defined by $\alpha_i = e_{i+1} - e_i$.
With this data we associate the affine Hecke algebra $\ahecke$ and its localized version $\lhecke$.
We identify $\aA$ with the ring $k[X_1^{\pm 1}, \dots, X_n^{\pm 1}]$, with $X_i$ corresponding to the exponential of $e_i$.
By \cite[Theorem 3.12]{rouquier-qha}, there is a map $\rhecke\to \lhecke$ given by mapping $x_i^\lambda\mapsto X_i \lambda_i^{-1} 1_\lambda$, and:
\begin{align*}
\tau_i^\lambda \mapsto \begin{cases}
\lambda_i X_{i+1}^{-1} (qX_i X_{i+1}^{-1} - 1)^{-1} (T_i - q)1_\lambda & \text{if $s_i(\lambda) = \lambda$}, \\
q^{-1} \lambda_i^{-1}X_{i+1} (X_i X_{i+1}^{-1} - 1)1_{s_i(\lambda)}T_i 1_\lambda & \text{if $\lambda(\alpha) = q$}, \\
\frac{X_i X_{i+1}^{-1} - 1}{qX_i X_{i+1}^{-1} - 1} 1_{s_i(\lambda)}T_i 1_\lambda & \text{else.}
\end{cases}
\end{align*}
Using our notation for $U_{x}$ the exponential of $x\in \Z^n$ in $\aA$ the group ring of $\Z^n$, we find the above mapping to be:
\begin{align*}
\tau_i^\lambda \mapsto \begin{cases}
\lambda_i X_{i+1}^{-1} (q U_{-\alpha_i} - 1)^{-1} (T_i - q)1_\lambda & \text{if $s_i(\lambda) =   \lambda$}, \\
q^{-1} \lambda_i^{-1}X_{i+1} (U_{-\alpha_i} - 1)1_{s_i(\lambda)}T_i 1_\lambda & \text{if         $\lambda(\alpha) = q$}, \\
\frac{U_{-\alpha_i} - 1}{qU_{-\alpha_i} - 1} 1_{s_i(\lambda)}T_i 1_\lambda & \text{else.}
\end{cases}
\end{align*}
It follows from the polynomial representations of $\lhecke, \qhecke$ that there is a mapping from  $\rhecke$ to $\qhecke$ given by mapping $x_i^\lambda\mapsto X_i^\lambda \lambda_i^{-1} 1_\lambda$, and:
\begin{align*}
\tau_i^\lambda \mapsto \begin{cases}
\lambda_i X_{i+1} r_{\alpha_i}^\lambda & \text{if $s_i(\lambda) = \lambda$}, \\
q^{-1} \lambda_i^{-1}X_{i+1} (U_{-\alpha_i} - 1) r_\alpha^\lambda & \text{if $\lambda(\alpha) = q$},\\
\frac{U_{-\alpha_i} - 1}{qU_{-\alpha_i} - 1} r_\alpha^\lambda & \text{else.}
\end{cases}
\end{align*}
Notice now, that the braid relation for $\tau_\alpha^\lambda$ is \emph{different} than the braid relation for $r_\alpha^\lambda$.

\begin{remark}
It should be noted that in \cite[Theorem 3.11]{rouquier-qha}, a grading is also given to the affine Hecke algebras of type $GL_n$ by use of an algebraic $W$-equivariant map $\dA\to \aA$. 
There is an algebraic obstruction to this approach in other types and we do not give gradings to affine Hecke algebras in this paper.
One could give irreducible representations of affine Hecke algebras gradings by using the correspondence given by \cite{lusztig} between their representation categories.

\end{remark}

\bibliographystyle{halpha}
\bibliography{refs}

\end{document}